\documentclass{elsarticle}
\usepackage{amsmath}
\usepackage{amsthm}
\usepackage{amssymb}
\usepackage{amscd}
\usepackage{stmaryrd}
\usepackage{amsfonts}
\usepackage{amsbsy}
\usepackage{epsfig,afterpage}
\usepackage[all]{xy}
\usepackage{graphicx}
\usepackage{color}
\usepackage{transparent}
\usepackage{subfigure}
\usepackage{import}
\usepackage[dvips]{psfrag}
\usepackage{color} 
\usepackage{enumerate}
\usepackage{overpic}
\usepackage{xcolor}
\usepackage{lineno}
\modulolinenumbers[5]
\usepackage{multicol}
\usepackage{lipsum}

\makeatletter
\DeclareFontFamily{U}{tipa}{}
\DeclareFontShape{U}{tipa}{m}{n}{<->tipa10}{}
\newcommand{\ark@char}{{\usefont{U}{tipa}{m}{n}\symbol{62}}}%

\newcommand{\ark}[1]{\mathpalette\ark@arc{$#1$}}

\newcommand{\ark@arc}[2]{%
	\sbox0{$\m@th#1#2$}%
	\vbox{
		\hbox{\resizebox{\wd0}{\height}{\ark@char}}
		\nointerlineskip
		\box0
	}%
}
\makeatother

\newtheorem {theorem} {Theorem} 

\newtheorem {proposition} {Proposition}

\newtheorem {lemma} {Lemma}
\newtheorem {definition} {Definition}
\newtheorem {remark} {Remark}

\begin{document}

\begin{frontmatter}
	
\title{Nonexistence results for a degenerate Goursat type Problem}

\author[1]{Olimpio Hiroshi Miyagaki}
\ead{olimpio@ufscar.br}
\author[2]{Carlos Alberto Reyes Peña\corref{cor1}}
\ead{carlos.reyes@estudante.ufscar.br}
\author[3]{Rodrigo da Silva Rodrigues}
\ead{rodrigorodrigues@ufscar.br}

\cortext[cor1]{Corresponding author}

\address[1]{olimpio@ufscar.br}
\address[2]{carlos.reyes@estudante.ufscar.br}
\address[3]{rodrigorodrigues@ufscar.br}




\begin{abstract}
For a generalization of the Gellerstedt operator with Dirichlet boundary conditions in a Tricomi domain. We establish Pohožaev-type identities and prove the nonexistence of nontrivial regular solutions. Furthermore, we investigate the critical exponent phenomenon for power-type nonlinearities, characterized by the critical exponent of a weighted Sobolev embedding.
\end{abstract}
\end{frontmatter}


\section{Introduction}\label{intro}

The mixed-type Dirichlet problem

\begin{equation}\label{PP}
\begin{cases} 
\mathcal{O}(u) :=-y^{m_1}u_{xx}-x^{m_2}u_{yy}=f(u) &\mbox{in}\quad \Omega, \\
\quad \ \ \  u=0  &\mbox{on}\quad \Gamma\subset\partial\Omega,
\end{cases}
\end{equation}
where $\Omega$ is a bounded, open, simply connected subset of $\mathbb{R}^2$ with piecewise $C^1$ boundary $\partial\Omega$, $f\in C^0(\mathbb{R})$ with $m_1, m_2\in\mathbb{N}.$

\vspace{0,3cm}
The significance of studying problem \eqref{PP} lies in the fact that several well-known operators can be seen as particular cases of the operator $\mathcal{O} = -y^{m_1}\partial^2_x - x^{m_2}\partial^2_y$ introduced in this paper. For instance, the Laplacian operator $-\Delta = -\partial^2_x - \partial^2_y$ as studied in \cite{brezisnirenberg}, \cite{brezis}, \cite{giltru}, the Tricomi operator $T = -y\partial^2_x - \partial^2_y$ in \cite{tricom}, \cite{lupopaynepopivanov}, and the Gellerstedt operator $L = -y^m\partial^2_x - \partial^2_y$ in \cite{gellerstedt}, \cite{lupopayne}, correspond to special cases when $m_1 = m_2 = 0$, $m_1 = 1$, $m_2 = 0$, and $m_2 = 0$, respectively. The operator $\mathcal{O}$ is a specific instance of the well-known Grushin operator $G = \Delta_x + (\alpha+1)^2 |x|^{2\alpha} \Delta_y$ in dimension 2, when $m_1 = 0$ and $\alpha \in \mathbb{N}$ (see \cite{montimorbidelli}, \cite{monti}, \cite{angelo}, \cite{liutangwang}, \cite{dousunwangzhu}). Nonexistence results for problems involving these operators are crucial in the theory of PDEs, particularly in the context of semilinear elliptic equations. Regarding the nonexistence of regular solutions for problems of the form

\begin{align}\label{PD}
\begin{cases} 
-\Delta u=-u_{xx}-u_{yy}=f(u) &\mbox{in}\quad \Omega, \\
\quad \;  \;u=0 &\mbox{on}\quad \partial\Omega,
\end{cases}
\end{align}
and
\begin{equation}\label{PL}
\begin{cases} 
L(u)=-y^{m}u_{xx}-u_{yy}=f(u) &\mbox{in}\quad \Omega, \\
\quad\ \;u=0, &\mbox{on}\quad \Gamma\subset\partial\Omega,
\end{cases}
\end{equation}

the critical growth plays an important and an crucial role. In \cite{pohozaev}, it was shown that problem \eqref{PD} does not admit any nontrivial sufficiently regular solutions in a bounded domain $\Omega\subset\mathbb{R}^n$ with $n\geq 3$ if $\Omega$ is star-shaped with respect to some interior point and if $f(u) = |u|^\alpha$ with $\alpha \geq 2^*-1 = \frac{n+2}{n-2}$, where $2^*$ is the well-known critical exponent in the Sobolev embedding $W^{1,2}(\Omega) \hookrightarrow L^p(\Omega)$ for $1 \leq p \leq 2^*$. Similarly, in \cite{lupopayne}, \cite{lupopaynepopivanov}, it was demonstrated that problem \eqref{PL} admits no nontrivial sufficiently regular solutions in a Tricomi domain $\Omega\subset\mathbb{R}^2$ if $\Omega$ is star-shaped with respect to the flow of $D=-(m+2)x\partial_x-2y\partial_y$, and if $f(u) = |u|^\alpha$ with $\alpha > 2^*(m) - 1 = \frac{m+8}{m}$, where $2^*(m)$ is the critical exponent in the weighted Sobolev embedding $W^{1,2}_m(\Omega) \hookrightarrow L^p(\Omega)$ for $1 \leq p < 2^*(m)$. Sabitova, in \cite{Sabitova}, studied an eigenvalue problem involving the Chaplygin-type operator 
$(sign \, y) |y|^m u_{xx} +u_{yy}-\lambda |y|^m u=0$, proving uniqueness and nonexistence results for $\lambda$ in certain ranges. These results, along with others found in \cite{brezisnirenberg}, \cite{brezis}, \cite{franchilanconelli}, \cite{montimorbidelli}, \cite{monti}, \cite{angelo}, \cite{dousunwangzhu}, among others, motivated the study of the nonexistence of regular solutions for problem \eqref{PP}, which is developed throughout this paper. To achieve this, we consider a family of nonhomogeneous dilations that maintain coordinate invariance in the solution of the homogeneous equation associated with problem \eqref{PP}, along with a weighted Sobolev space naturally related to the operator $\mathcal{O}$. In Section \ref{Sec2}, we demonstrate the critical exponent phenomenon in the embedding of this space with respect to dilation.

\vspace{0,3cm}

Another relevant result presented in this article is the construction of the Poho\^{z}aev-type identity for problem \eqref{PP} inspired by the classic argument used for problem \eqref{PD} on a bounded domain found in \cite{pohozaev} and recovered in \cite{lupopayne} for problem \eqref{PL} on a class of domains related to the operator (Tricomi domains for the operator) and also in \cite{lupopaynepopivanov} for the problem \eqref{PL} with $m=1$. Other more recent works for the Grushin operator involving Poho\^zaev-type identities are \cite{liutangwang} and \cite{pengwangyang}. In addition, we explore some domain variations for problem \eqref{PP} as presented in Section \ref{Sec3}. Finally, we prove the main results of this paper in Section \ref{Sec4}, where we conclude by using the Poho\^{z}aeh-type identity and the Hardy-Sobolev inequality that the problem \eqref{PP} on the domains of the Section \ref{Sec3} do not have sufficiently regular nontrivial solutions. 

\vspace{0,3cm}
During of the writing of this article we find several difficulties, some of which are practical difficulties such as knowing of the domains of the solution, because $m_1$ and $m_2$ are arbitrary but need to have specific conditions for its existence. Another difficulty, as in every research, is the standard computations; in our case, not only the exponents of the operator, also the term $x^{m_2}$ influenced the search for an other more particular domain to be able to overcome these obstacles, see Theorem \ref{T6}. Nonetheless, the primary challenge was working with weighted Sobolev spaces, we needed to define a weighted gradient $\nabla_{m_1,m_2}u=(|y|^{\frac{m_1}{2}}u_x,|x|^{\frac{m_2}{2}}u_y)$ and the operator $\mathcal{X}u=(-y^{m_1}u_x,-x^{m_2}u_y)$. We observed that the relationship $\Delta u=div(\nabla u)$ and the norm $||u||_{H^1(\Omega)}^2=||\nabla u||_{L^2(\Omega)}^2+||u||_{L^2(\Omega)}^2$ are crucial for studying nonexistence results, as the norm of the solution space and the operator depend on the gradient. Note that, when $m_1=m_2=0$ we get $\nabla_{m_1,m_2} u=\nabla u$ and $\mathcal{X}u=(-u_x,-u_y)=-\nabla u$. Notice that, $\nabla_{m_1,m_2} u=-\mathcal{X}u$ when $m_1=m_2=0$. However, in general, this relationship does not hold. In conclusion, the norm of the solution space depends on the weighted gradient, which differs from the operator. Specifically, $||u||_{H^1_{m_1,m_2}}^2=||\nabla_{m_1,m_2} u||_{L^2(\Omega)}^2+||u||_{L^2(\Omega)}^2
$, while the operator $\mathcal{O}u=div(\mathcal{X}u)\neq div(-\nabla_{m_1,m_2} u)$.

\vspace{0,3cm}
We can not end the introduction without first announcing some of the important results found in this paper.

\begin{proposition}
Let $m_1,m_2\in\mathbb{N}$ given. Suppose there exists $C>0$ independent of $u$ such that
\begin{equation*}
    ||u||_{L^p(\Omega)}\leq C||\nabla_{m_1,m_2}u||_{L^2(\Omega)}, \ \ \forall \ u\in C^\infty_0(\Omega).
\end{equation*}
Then, $1\leq p\leq 2^*(m_1,m_2):=\frac{2(m_1+m_2)+8}{m_1+m_2+m_1m_2}\cdot$
\end{proposition}

In the following results, we consider $\Omega_1$ to be an open, bounded, simply connected set with a piecewise $C^1$ boundary, $\partial \Omega_1 = \sigma_1 \cup AC \cup BC$, formed by an arc $\sigma_1 \subset {(x,y) \in \mathbb{R}^2; \ y > 0}$, which intercepts the x-axis at points $A = (2x_0, 0)$ and $B = (0, 0)$, where $x_0 < 0$, and by the characteristic curves $AC$ and $BC$ in ${(x,y) \in \mathbb{R}^2; \ x \leq 0 \ \text{and} \ y \leq 0}$ of the operator $\mathcal{O}$, passing through points $A$ and $B$, respectively, and meeting at point $C$. We say that $\Omega_1$ is a Tricomi domain for the operator $\mathcal{O}$.

\begin{theorem}
Let $\Omega_1$ be a Tricomi domain for the operator $\mathcal{O}$ with $m_1$, $\frac{m_2}{2}$ odd and even numbers respectively, whose boundary $\partial\Omega_1$ is piecewise $C^1$ with the exterior unit normal vector $\eta$. If $u\in C^2(\overline{\Omega_1})$ is a solution of the problem \eqref{P1}, then 

\begin{align*}
    (m_1+m_2+4)\int_{\Omega_1} F(u)-\frac{m_1+m_2+m_1m_2}{2}\int_{\Omega_1} uf(u)= \nonumber\frac{1}{2}\left[\int_{BC\cup\sigma_1}\omega_1+\int_{BC}\omega_2\right],
\end{align*}
where $F$ is the primitive function of $f\in C^0(\mathbb{R})$ such that $F(0)=0$, $\omega_1$ and $\omega_2$ are defined by 

\begin{align*}
    \omega_1=\Bigl[2Du(-y^{m_1}u_x,-x^{m_2}u_y) +(y^{m_1}u^2_x+x^{m_2}u^2_y)(-(m_1+2)x,-(m_2+2)y\Bigl]\cdot\eta \, ,
\end{align*}
\begin{align*}
    \omega_2=\Bigl[-2F(u)(-(m_1+2)x,-(m_2+2)y)
    -(m_1+m_2+m_1m_2)u (-y^{m_1}u_x,-x^{m_2}u_y)\Bigl]\cdot\eta \, ,
\end{align*}
and $D$ is the vector field
\begin{align*}
    D=-(m_1+2)x\partial_x-(m_2+2)y\partial_y.
\end{align*}
\end{theorem}

\begin{theorem}
Let $\Omega_1$ be a Tricomi domain for the operator $\mathcal{O}$, which is $D$-star-shaped where $D=-(m_1+2)x\partial_x-(m_2+2)y\partial_y$ with $m_1$, $\frac{m_2}{2}$ odd and even respectively. Let $u\in C^2(\overline{\Omega_1})$ be a regular solution of
\begin{equation*}
 \begin{cases}
\mathcal{O}(u):=-y^{m_1}u_{xx}-x^{m_2}u_{yy}=u|u|^{\alpha-1} &\mbox{in} \quad\Omega_1, 
\\ \quad \ \ \;u=0 & \mbox{on} \quad AC\cup\sigma_1\subseteq\partial\Omega_1,
\end{cases}
\end{equation*}
with $\alpha>2^*(m_1,m_2)-1=\frac{m_1+m_2-m_1m_2+8}{m_1+m_2+m_1m_2}$. Then $u\equiv0$ a.e. in $\Omega_1$.
\end{theorem}

Other similar theorems involving different domains can also be found throughout this article; this brings with it challenges, particularly in the necessary estimates arising from the domain's geometry, which had to be overcome in the proofs through a refinement of the applied technique.


\section{Statement of the results}\label{Sec2}

In this section, we will study the invariance of the solutions of the homogeneous equation of problem \eqref{PP} determined by a family of nonhomogeneous dilations that determines an infinitesimal generator $D$. We will define $D$-star-shaped domains and the weighted Sobolev space associated to problem \eqref{PP} and we conclude with a critical exponent phenomenon of Sobolev-Gagliardo-Nirenberg type over $D$-star-shaped domains.

\vspace{0,3cm}

Define $\phi_\lambda(x,y)=(\lambda^{-\alpha}x,\lambda^{-\beta}y)$, for $\alpha,\beta >0$ with $\lambda>0$, a family of nonhomogeneous dilations (i.e. $\alpha\neq\beta$) that generates $\psi_\lambda(u)=u\circ\phi_\lambda=u_\lambda$, a family of nonhomogeneous dilations operators and its infinitesimal generator

\begin{equation}
    \left[\frac{d}{d\lambda}u_\lambda\right]_{\mid _{\lambda=1}}=Du=-\alpha x u_x-\beta y u_y.
\end{equation}

The flow $\mathcal{F}_t:\mathbb{R}^2\rightarrow\mathbb{R}^2$ defined by $\mathcal{F}_t(x_0,y_0)=(x(t),y(t))$, for each $t\in\mathbb{R}$, being the unique integral curve of $D=-\alpha x\partial_x-\beta y\partial_y=(-\alpha x,-\beta y)\cdot\nabla$ passing through the point $(x_0,y_0)$ in time $t=0$. Note that $(x(t),y(t))=(x_0e^{-\alpha t},y_0e^{-\beta t})$, because, $\gamma'(t)=(x'(t),y'(t))=(-\alpha x(t),-\beta y(t))=D(\gamma(t))$. Hence $x'(t)=-\alpha x(t)$ and $y'(t)=-\beta y(t)$, this is, $$\frac{d}{dt}x(t)=-\alpha x(t) \ \ \text{and}  \ \ \frac{d}{dt}y(t)=-\beta y(t).$$ Therefore, $x(t)=C_1e^{-\alpha t}$ and $y(t)=C_2e^{-\beta t}$ replacing $t=0$, we found $x(0)=x_0=C_1$ and $y(0)=y_0=C_2$, so $$(x(t),y(t))=(x_0e^{-\alpha t},y_0e^{-\beta t}).$$

Furthermore, 
\begin{align*}
    \mathcal{F_{+\infty}}(x_0,y_0)&=\lim_{t\rightarrow+\infty}\mathcal{F}_t(x_0,y_0)=\lim_{t\rightarrow+\infty}(x(t),y(t)) \\
    &=\lim_{t\rightarrow+\infty}(x_0e^{-\alpha t},y_0e^{-\beta t})=(0,0),
\end{align*}
for each starting point $(x_0,y_0)\in\mathbb{R}^2$. In conclusion, $\mathbb{R}^2$ is star-shaped with respect to the origin using the field $D$. We said $\mathbb{R}^2$ is $D$-star-shaped in the sense of the following definition.

\begin{definition}
Let $\alpha,\beta>0$. A open set $\Omega\subset\mathbb{R}^2$ is said $D$-star-shaped (or star-shaped with respect to the field $D=-\alpha x\partial_x-\beta y\partial_y$) if for each $(x_0,y_0)\in\overline{\Omega}$ one has to $\mathcal{F}_t(x_0,y_0)\subset\overline{\Omega}$ for each $t\in [0,+\infty]$.
\end{definition}

Given $D$-star-shaped property in $\Omega$, the $D$-starlike property on $\partial\Omega$ and the relationship between them mentioned in the following lemma can be found in \cite{lupopayne}.

\begin{lemma}
Let $\Omega$ be an open set with piecewise $C^1$ boundary $\partial\Omega$. If $\Omega$ is $D$-star-shaped where $D=-\alpha x\partial_x-\beta y\partial_y$ with $\alpha, \beta>0$,
then $\partial\Omega$ is starlike, that is,
\begin{equation}\label{SL}
    (\alpha x,\beta y)\cdot\eta(x,y) \geq0,
\end{equation}
for each regular point $(x,y)\in\partial\Omega$ where $\eta(x,y)$ is the unit exterior normal vector on $\partial\Omega$ at the point $(x,y)$. 
\end{lemma}

\begin{definition}\label{OPIE}
A subset $\Gamma\subset\partial\Omega$ is said $D$-starlike if the condition \eqref{SL} holds on $\Gamma$, which is equivalent to the condition 
\begin{equation}\label{DSL}
    \alpha x dy- \beta y dx \geq0 \ \ \mbox{on} \ \ \Gamma,
\end{equation}
where $\partial\Omega$ is given by the positive orientation leaving the interior of $\Omega$ on the left side.
\end{definition}

We want to study the behavior of regular solutions of the homogeneous problem associated with the operator 
\begin{equation*}
    \mathcal{O}=-y^{m_1}\partial^2_x-x^{m_2}\partial^2_y \ \mbox{where} \ m_1, m_2\in\mathbb{N},
\end{equation*}
with respect to certain nonhomogeneous dilations on bounded domains. Without loss of generality, for bounded domain $\Omega$ of $\mathbb{R}^2$ we will consider the family of nonhomogeneous dilations 

\begin{equation}
    \phi_\lambda(x,y)=(\lambda^{-m_1-2}x,\lambda^{-m_2-2}y), \ \mbox{with} \ \lambda\in(0,1].
\end{equation}

 If $u$ is a regular solution of homogeneus problem $\mathcal{O}(u)=0$ in $\Omega$, then $u_\lambda=u\circ\phi_\lambda$ is a regular solution of homogeneus problem too. This is, $\mathcal{O}(u_\lambda)=0$. Since, $u_\lambda(x,y)=u(\lambda^{-m_1-2}x,\lambda^{-m_2-2}y)$, we obtain, 
 
 $$\frac{\partial}{\partial x}u_\lambda(x,y)=\lambda^{-m_1-2}u_x(\lambda^{-m_1-2}x,\lambda^{-m_2-2}y)$$
 and
 $$\frac{\partial}{\partial y}u_\lambda(x,y)=\lambda^{-m_2-2}u_y(\lambda^{-m_1-2}x,\lambda^{-m_2-2}y).$$
 
 Moreover, 
 $$\frac{\partial^2}{{\partial x}^2}u_\lambda(x,y)=\lambda^{-2m_1-4}u_{xx}(\lambda^{-m_1-2}x,\lambda^{-m_2-2}y)$$
 and
 $$\frac{\partial^2}{{\partial y}^2}u_\lambda(x,y)=\lambda^{-2m_2-4}u_{yy}(\lambda^{-m_1-2}x,\lambda^{-m_2-2}y).$$

Then, 
\begin{align*}
    \mathcal{O}(u_\lambda(x,y))=&-y^{m_1}\lambda^{-2m_1-4}u_{xx}(\lambda^{-m_1-2}x,\lambda^{-m_2-2}y)\\ &-x^{m_2}\lambda^{-2m_2-4}u_{yy}(\lambda^{-m_1-2}x,\lambda^{-m_2-2}y) \\
    =&\frac{\lambda^{-4}}{\lambda^{-m_1m_2}}[-y^{m_1}\lambda^{-2m_1-m_1m_2}u_{xx}(\lambda^{-m_1-2}x,\lambda^{-m_2-2}y) \\ 
    &-x^{m_2}\lambda^{-2m_2-m_1m_2}u_{yy}(\lambda^{-m_1-2}x,\lambda^{-m_2-2}y)].
\end{align*}

The change of variables $\overline{x}=\lambda^{-m_1-2}x$ and $\overline{y}=\lambda^{-m_2-2}y$ give us $$\overline{x}^{m_2}=\lambda^{-m_1m_2-2m_2}x^{m_2} \ \ \mbox{and} \ \ \overline{y}^{m_1}=\lambda^{-m_1m_2-2m_1}y^{m_1}.$$ 

By hypotheses, $\mathcal{O}(u(\overline{x},\overline{y}))=-\overline{y}^{m_1}u_{xx}(\overline{x},\overline{y})-\overline{x}^{m_2}u_{yy}(\overline{x},\overline{y})=0$ in $\Omega$. Therefore, $$\mathcal{O}(u_\lambda(x,y))=\frac{\lambda^{-4}}{\lambda^{-m_1m_2}}\mathcal{O}(u(\overline{x},\overline{y}))=0.$$

Consider the following weighted Sobolev space $H^1_{m_1,m_2}(\Omega)$, associated with the problem \eqref{PP}, as the completion of $C^\infty(\Omega)$ with respect to the norm 
\begin{equation}\label{N1}
    ||u||_{H^1_{m_1,m_2}(\Omega)}:=\left[\int_\Omega|y|^{m_1}u_x^2+|x|^{m_2}u_y^2+u^2\right]^{\frac{1}{2}}
\end{equation}
and the subset $H^1_{m_1,m_2,0}(\Omega)=\overline{C^\infty_0(\Omega)}$ with the norm of $H^1_{m_1,m_2} (\Omega)$ above.

\vspace{0,3cm}

We will prove that actually \eqref{N1} defines a norm for the spaces $H^1_{m_1,m_2}(\Omega)$. First, we define the weighted gradient  $$\nabla_{m_1,m_2}u:=\left(|y|^{\frac{m_1}{2}}u_x,|x|^{\frac{m_2}{2}}u_y\right).$$ 

So $$||\nabla_{m_1,m_2}u||^2_{L^2(\Omega)}=\int_\Omega\left(|y|^{m_1}u_x^2+|x|^{m_2}u_y^2\right)$$ and with this we rewrite $||\cdot||_{H^1_{m_1,m_2}(\Omega)}$ by 
\begin{align}\label{N2}
    ||u||^2_{H^1_{m_1,m_2}(\Omega)}=||\nabla_{m_1,m_2}u||^2_{L^2(\Omega)}+||u||^2_{L^2(\Omega)}.
\end{align}           

In fact, If $||u||_{H^1_{m_1,m_2}(\Omega)}=0$ then, $||\nabla_{m_1,m_2}u||^2_{L^2(\Omega)}+||u||^2_{L^2(\Omega)}=0$ so, $u=0$ a.e. in $\Omega$. Moreover, if $u=0$ a.e. in $\Omega$ then, $||u||^2_{L^2(\Omega)}=0$ and $u_x=0$ and $u_y=0$ a.e. in $\Omega$. So, $\nabla_{m_1,m_2}u=(0,0)$. Therefore, $||u||_{H^1_{m_1,m_2}(\Omega)}=0$. Note that, 
\begin{align*}
    |\lambda|\; ||u||_{H^1_{m_1,m_2}(\Omega)}&=|\lambda|\; ||\nabla_{m_1,m_2}u||_{L^2(\Omega)}+|\lambda|\;||u||_{L^2(\Omega)} \\
    &=||\nabla_{m_1,m_2}(\lambda u)||_{L^2(\Omega)}+||\lambda u||_{L^2(\Omega)} \\
    &=||\lambda u||_{H^1_{m_1,m_2}(\Omega)},
\end{align*} 
for all $u\in H^1_{m_1,m_2}(\Omega)$ and for all $ \lambda\in\mathbb{R}$.

To finish, just prove $$||u+v||^2_{H^1_{m_1,m_2}(\Omega)}\leq\Bigl(||u||_{H^1_{m_1,m_2}(\Omega)}+||v||_{H^1_{m_1,m_2}(\Omega)}\Bigl)^2$$ for any $u,v\in H^1_{m_1,m_2}(\Omega)$.

Observe that, $\nabla_{m_1,m_2}(u+v)=\nabla_{m_1,m_2}u+\nabla_{m_1,m_2}v$. Hence,
\begin{align*}
    ||u+v||^2_{H^1_{m_1,m_2}(\Omega)}\leq& \Bigl(||\nabla_{m_1,m_2}u||_{L^2(\Omega)}+||\nabla_{m_1,m_2}v||_{L^2(\Omega)}\Bigl)^2+\Bigl(||u||_{L^2(\Omega)}+||v||_{L^2(\Omega)}\Bigl)^2 \\
    =&||u||^2_{H^1_{m_1,m_2}(\Omega)}+||v||^2_{H^1_{m_1,m_2}(\Omega)} \\ &+2||\nabla_{m_1,m_2}u||_{L^2(\Omega)}||\nabla_{m_1,m_2}v||_{L^2(\Omega)}+2||u||_{L^2(\Omega)}||v||_{L^2(\Omega)}.
\end{align*}

On the other hand, 
\begin{align*}
    \Bigl(||u||_{H^1_{m_1,m_2}(\Omega)}+||v||_{H^1_{m_1,m_2}(\Omega)}\Bigl)^2=&||u||^2_{H^1_{m_1,m_2}(\Omega)}+||v||^2_{H^1_{m_1,m_2}(\Omega)}\\ &+2||u||_{H^1_{m_1,m_2}(\Omega)}||v||_{H^1_{m_1,m_2}(\Omega)}.
\end{align*}

All that's left is to analyze, $$||\nabla_{m_1,m_2}u||_{L^2(\Omega)}||\nabla_{m_1,m_2}v||_{L^2(\Omega)}+||u||_{L^2(\Omega)}||v||_{L^2(\Omega)}\leq ||u||_{H^1_{m_1,m_2}(\Omega)}||v||_{H^1_{m_1,m_2}(\Omega)}.$$

In fact,
\begin{align*}
\Bigl(||\nabla_{m_1,m_2}u||_{L^2(\Omega)}&||\nabla_{m_1,m_2}v||_{L^2(\Omega)}+||u||_{L^2(\Omega)}||v||_{L^2(\Omega)}\Bigl)^2\\=
&||\nabla_{m_1,m_2}u||^2_{L^2(\Omega)}||\nabla_{m_1,m_2}v||^2_{L^2(\Omega)}\\
&+2||\nabla_{m_1,m_2}u||_{L^2(\Omega)}||\nabla_{m_1,m_2}v||_{L^2(\Omega)}||u||_{L^2(\Omega)}||v||_{L^2(\Omega)}\\
&+||u||^2_{L^2(\Omega)}||v||^2_{L^2(\Omega)},
\end{align*} 
and 
\begin{align*}
    ||u||^2_{H^1_{m_1,m_2}(\Omega)}&||v||^2_{H^1_{m_1,m_2}(\Omega)}\\
=&\bigl(||\nabla_{m_1,m_2}u||^2_{L^2(\Omega)}+||u||^2_{L^2(\Omega)}\bigl)\bigl(||\nabla_{m_1,m_2}v||^2_{L^2(\Omega)}+||v||^2_{L^2(\Omega)}\bigl) \\
    =&||\nabla_{m_1,m_2}u||^2_{L^2(\Omega)}||\nabla_{m_1,m_2}v||^2_{L^2(\Omega)}+||\nabla_{m_1,m_2}u||^2_{L^2(\Omega)}||v||^2_{L^2(\Omega)}\\ &+||u||^2_{L^2(\Omega)}||\nabla_{m_1,m_2}v||^2_{L^2(\Omega)}+||u||^2_{L^2(\Omega)}||v||^2_{L^2(\Omega)}.
\end{align*}

Pull all together, it follows that $||\cdot||_{H^1_{m_1,m_2}(\Omega)}$ is a norm in $H^1_{m_1,m_2}(\Omega)$. The next result deals with the behavior of critical exponent which provides an estimate for this, if there is an immersion of type Sobolev-Gagliardo-Nirenberg on $H^1_{m_1,m_2,0}(\Omega).$ Its proof follows using a similar idea made in \cite{lupopayne}.

\begin{proposition}\label{IMER}
Let $m_1,m_2\in\mathbb{N}$ given. Suppose there exists $C>0$ independent of $u$ such that
\begin{equation}\label{SGN}
    ||u||_{L^p(\Omega)}\leq C||\nabla_{m_1,m_2}u||_{L^2(\Omega)}, \ \ \forall \ u\in C^\infty_0(\Omega).
\end{equation}
Then, $1\leq p\leq 2^*(m_1,m_2):=\frac{2(m_1+m_2)+8}{m_1+m_2+m_1m_2}\cdot$
\end{proposition}

\begin{proof}
    If $u=0$, \eqref{SGN} is trivially held. Take $\ u\in C^\infty_0(\Omega)-\{0\}$ arbitrary. Using the change of variables theorem gives 
$$||u_\lambda||^p_{L^p(\Omega)}=\lambda^{m_1+m_2+4}||u||^p_{L^p(\Omega)}$$ and $$||\nabla_{m_1,m_2}u_\lambda||^2_{L^2(\Omega)}=\lambda^{m_1+m_2+m_1m_2}||\nabla_{m_1,m_2}u||^2_{L^2(\Omega)}.$$

\vspace{0,3cm}

Replacing $u_\lambda$ em \eqref{SGN}, we have

\begin{align*}
   ||u||_{L^p(\Omega)}&=\lambda^{-\frac{m_1+m_2+4}{p}}||u_\lambda||_{L^p(\Omega)}\leq C \lambda^{-\frac{m_1+m_2+4}{p}}||\nabla_{m_1,m_2}u_\lambda||_{L^2(\Omega)} \\
   &= C \lambda^{\left(-\frac{m_1+m_2+4}{p}+\frac{m_1+m_2+m_1m_2}{2}\right)}||\nabla_{m_1,m_2}u||_{L^2(\Omega)}.
\end{align*}

Denote $r=-\frac{m_1+m_2+4}{p}+\frac{m_1+m_2+m_1m_2}{2}$ and suppose that $r>0$. Then, making $\lambda\rightarrow0^+$, we get $||u||_{L^p(\Omega)}\leq 0$ but $u\neq0$. In conclusion, 

$$r=-\frac{m_1+m_2+4}{p}+\frac{m_1+m_2+m_1m_2}{2}\leq0,$$
 then
 $$p\leq \frac{2(m_1+m_2)+8}{m_1+m_2+m_1m_2}\cdot$$\end{proof}


\section{Tricomi domains and Poho\^{z}aev-type identities}\label{Sec3}

We will study the problem \eqref{PP} on three different domains in this paper, and we will construct Poho\^{z}aev-type identities for them. We start with the  problem

\begin{equation}\label{P1}
\begin{cases}
  \mathcal{O}(u):=-y^{m_1}u_{xx}-x^{m_2}u_{yy}=f(u)  &\mbox{in}\quad \Omega_1, \\
 \quad \ \ \;u=0   &\mbox{on}\quad AC\cup\sigma_1\subseteq\partial\Omega_1,
\end{cases}
\end{equation}
where $\Omega_1$ is a Tricomi domain for the operator $\mathcal{O}=-y^{m_1}\partial^2_x-x^{m_2}\partial^2_y$. That mean, $\Omega_1$ is open, bounded, simply connected set with piecewise $C^1$ boundary $\partial\Omega_1=\sigma_1\cup AC\cup BC$ formed by an elliptical arc $\sigma_1\subset\{(x,y)\in\mathbb{R}^2; \ y>0\}$ intercepting the axis $x$ at the points $A=(2x_0,0)$ and $B=(0,0)$ with $x_0<0$ and by the characteristics curves $AC$ and $BC$ in $\{(x,y)\in\mathbb{R}^2; \ x\leq0 \ \mbox{and} \ y\leq0\}$ of $\mathcal{O}$ passing through points $A$ and $B$ respectively, which meet at point $C$. See figure \ref{Sorv1}.

\begin{figure}[!ht]
    \centering
    \includegraphics[width=8cm,angle=-90]{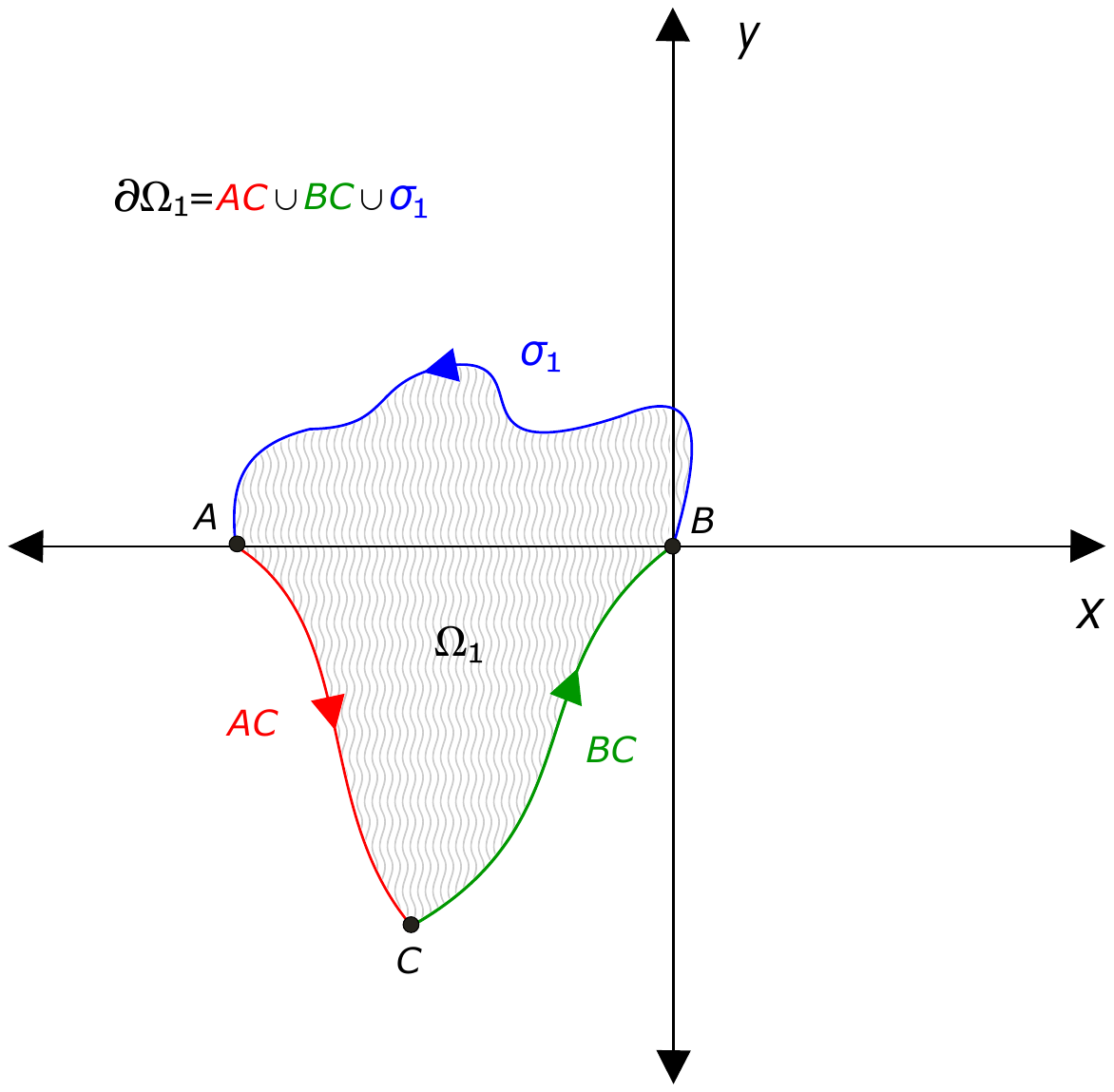}
    \caption{Domain $\Omega_1$ for the problem \eqref{P1}.}
    \label{Sorv1}
\end{figure}

\begin{remark}\label{OO}
The domain $\Omega_1$ depends on $m_1, m_2$ because the characteristics curves are associated with the operator $\mathcal{O}$, additionally $\Omega_1$ depends on $x_0$ which determine the parabolic diameter $2|x_0|=|AB|$ and its position in $\mathbb{R}^2$. This is why, we consider the problems \eqref{P2} and \eqref{P3}.
\end{remark}

\begin{definition}\label{CC}
Let $\Gamma(t)=(\alpha(t),\beta(t))$ be a smooth $(C^\infty)$ curve in $\mathbb{R}^2$. We say $\Gamma$ is a characteristics curve of equation 
$$au_{xx}+2bu_{xy}+cu_{yy}=d,$$
if $a(\beta'(t))^2-2b(\alpha'(t))(\beta'(t))+c(\alpha'(t))^2=0$. 
\end{definition}

\vspace{0,3cm}

For the problem \eqref{P1}, $a=-y^{m_1}$, $b=0$ and $c=-x^{m_2}$. By the Definition \ref{CC} its characteristic curves satisfies 
\begin{equation}\label{CCC}
    (-y^{m_1})(\beta'(t))^2+(-x^{m_2})(\alpha'(t))^2=0.
\end{equation}

We know $\Gamma$ can be written locally as the graph of a function because, $\Gamma$ is a smooth curve. Suppose $\Gamma(x)=(x,y(x))$, in this case, $\alpha(x)=x$ and $\beta(x)=y(x)$ so $\alpha'(x)=1$ and $\beta'(x)=y'(x)$. By \eqref{CCC}, we concluded that the characteristics curves of the problem \eqref{P1} are solutions of 
$$-y^{m_1}(y'(x))^2=x^{m_2}.$$

To solve this ordinary differential equation in $\{(x,y)\in\mathbb{R}^2; \ x\leq0 \ \mbox{and} \ y\leq0\}$, without loss of generality, consider $m_1\in\mathbb{N}$ odd and $\frac{m_2}{2}\in\mathbb{N}$ even . On this conditions the characteristics curves of \eqref{P1} satisfy  
$$(-y)^\frac{m_1}{2}|y'(x)|=x^\frac{m_2}{2}.$$

We will analyze $|y'(x)|$.

\begin{multicols}{3}
\begin{center}
\textbf{case 1}
     $$y'(x)<0$$
    $$(-y)^\frac{m_1}{2}(-y'(x))=x^\frac{m_2}{2}$$
\end{center}
\columnbreak
\begin{center}
\textbf{case 2}
   $$y'(x)=0$$
    $$y(x)=K_3$$
\end{center}
\raggedcolumns
\columnbreak
\begin{center}
\textbf{case 3}
    $$y'(x)>0$$
    $$(-y)^\frac{m_1}{2}y'(x)=x^\frac{m_2}{2}$$
\end{center}
\end{multicols}

In case 1, we get $\frac{2}{m_1+2}(-y)^\frac{m_1+2}{2}=\frac{2}{m_2+2}x^\frac{m_2+2}{2}+K_1$. In case 2, we get a contradiction with the choice of $\Gamma(x)=(x,y(x))$. In case 3, we get $\frac{2}{m_1+2}(-y)^\frac{m_1+2}{2}=-\frac{2}{m_2+2}x^\frac{m_2+2}{2}+K_2$. 

\vspace{0,3cm}

In order to construct $\partial\Omega_1$ and consequently $\Omega_1$, it is enough to substitute in the characteristic curves found the points $A=(2x_0,0)$ and $B=(0,0)$. In relation to Definition \ref{OPIE} we give the orientation for $\Omega_1$ and particularly for $AC$ when $y'(x)<0$ and for $BC$ when $y'(x)>0$.  We have

$$k_1=-\frac{2}{m_2+2}(2x_0)^\frac{m_2+2}{2} \ \ \mbox{and} \ \  k_2=0.$$

The characteristic curves $AC$ and $BC$ in $\{(x,y)\in\mathbb{R}^2; \ x\leq0 \ \mbox{and} \ y\leq0\}$ of the operator $\mathcal{O}$, for $m_1\in\mathbb{N}$ odd and $\frac{m_2}{2}\in\mathbb{N}$ even, are given by
\begin{align}\label{AC1}
    AC=\bigg\{(x,y)\in\mathbb{R}^2;& \ y_c\leq y\leq0, \nonumber\\
&\frac{2}{m_1+2}(-y)^\frac{m_1+2}{2}=\frac{2}{m_2+2}\left[x^\frac{m_2+2}{2}-(2x_0)^\frac{m_2+2}{2}\right] \bigg\},
\end{align}
\begin{align}\label{BC1}
BC=\bigg\{ (x,y)\in\mathbb{R}^2; \ y_c\leq y\leq0, \ \frac{2}{m_1+2}(-y)^\frac{m_1+2}{2}=-\frac{2}{m_2+2}x^\frac{m_2+2}{2} \bigg\},
\end{align}
and the point $C=\left((\frac{1}{2})^\frac{2}{m_2+2}(2x_0), y_c\right)$ where $$y_c=-\left[-\frac{1}{2}\left(\frac{m_1+2}{m_2+2}\right)(2x_0)^\frac{m_2+2}{2}\right]^\frac{2}{m_1+2}.$$

Parameterizing \eqref{AC1} and \eqref{BC1} with respedt to the variable $y$, the exterior normal vector will be given by
\begin{align}\label{nAC1}
\eta_{AC}=\left(-1,-\left[\frac{m_2+2}{m_1+2}(-y)^\frac{m_1+2}{2}+(2x_0)^\frac{m_2+2}{2}\right]^\frac{-m_2}{m_2+2}(-y)^\frac{m_1}{2}\right),
\end{align}
and
\begin{align}\label{nBC1}
\eta_{BC}=\left(1,-\left[-\frac{m_2+2}{m_1+2}(-y)^\frac{m_1+2}{2}\right]^\frac{-m_2}{m_2+2}(-y)^\frac{m_1}{2}\right),
\end{align}
respectively. 


As mentioned in Remark \ref{OO}, we will consider the problem \eqref{P2} and its respective characteristic curves which are found analogously to problem \eqref{P1}, namely 

\begin{equation}\label{P2}
\begin{cases}
  \mathcal{O}(u):=-y^{m_1}u_{xx}-x^{m_2}u_{yy}=f(u) &\mbox{in}\quad \Omega_2, \\
\quad \ \ \;u=0 &\mbox{on}\quad AC\cup\sigma_2\subseteq\partial\Omega_2,
\end{cases}
\end{equation}
where $\Omega_2$ is a Tricomi domain for the operator $\mathcal{O}=-y^{m_1}\partial^2_x-x^{m_2}\partial^2_y$. That is, $\Omega_2$ is open, bounded, simply connected set with piecewise $C^1$ boundary $\partial\Omega_2=\sigma_2\cup AC'\cup BC'$ formed by an elliptical arc $\sigma_2\subset\{(x,y)\in\mathbb{R}^2; \ y>0\}$ intercepting the axis $x$ at the points $A'=(2x_0,0)$ and $B'=(0,0)$ with $x_0>0$ and the characteristics curves $AC'$ and $BC'$ in $\{(x,y)\in\mathbb{R}^2; \ x\geq0 \ \mbox{and} \ y\leq0\}$ of $\mathcal{O}$ passing through points $A'$ and $B'$ respectively and meet at point $C'$. See figure \ref{Sorv2}.

\begin{figure}[!ht]
    \centering
    \includegraphics[width=8cm,angle=-90]{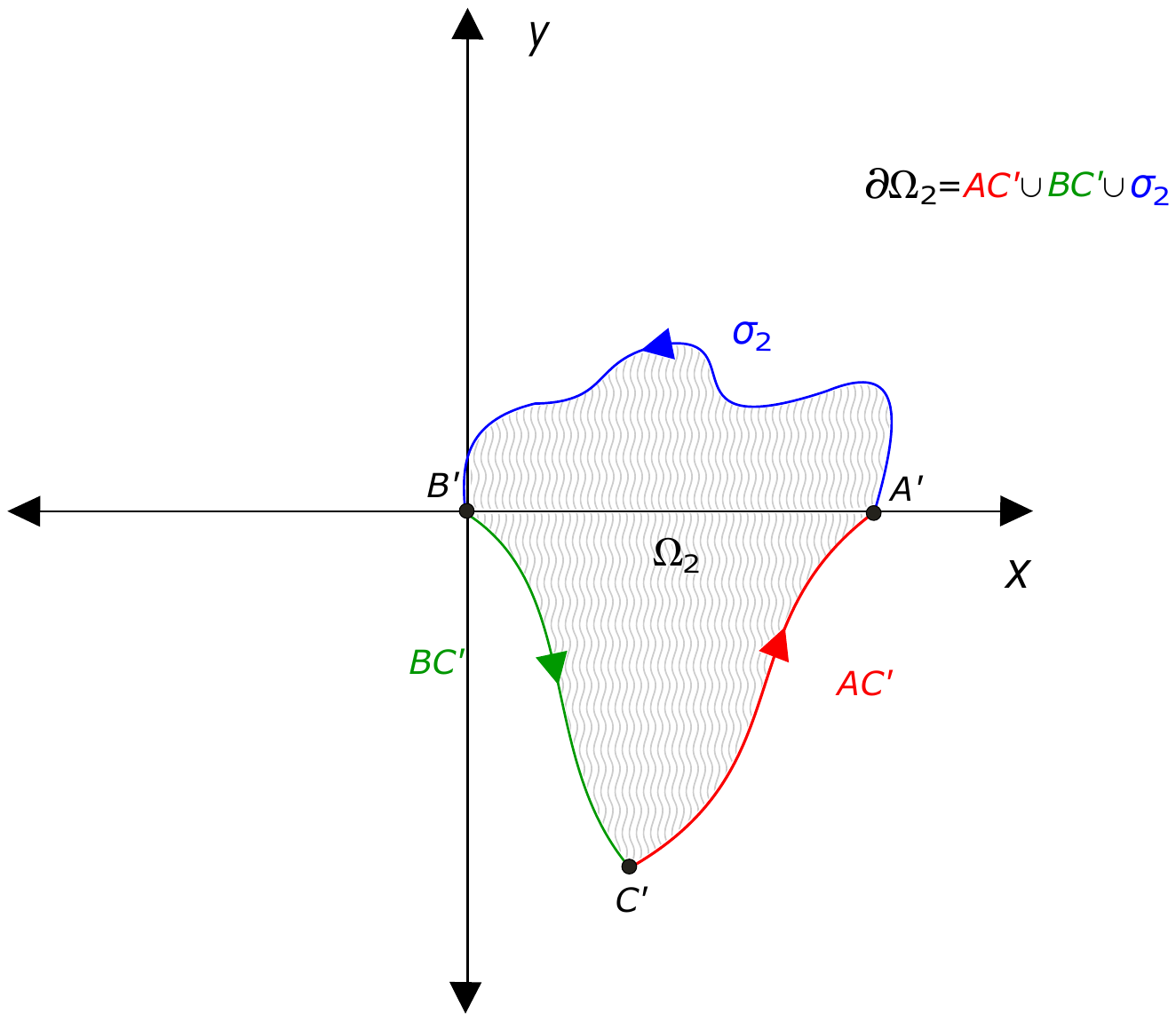}
    \caption{Domain $\Omega_2$ for the problem \eqref{P2}.}
    \label{Sorv2}
\end{figure}

The characteristics curves $AC'$ and $BC'$ in $\{(x,y)\in\mathbb{R}^2; \ x\geq0 \ \mbox{and} \ y\leq0\}$ of the operator $\mathcal{O}$ for $m_1\in\mathbb{N}$ odd number and $m_2\in\mathbb{N}$, are given by
\begin{align}\label{AC2}
    AC'=\bigg\{ (x,y)\in&\mathbb{R}^2; \ y_c\leq y\leq0, \nonumber\\
&\frac{2}{m_1+2}(-y)^\frac{m_1+2}{2}=-\frac{2}{m_2+2}\left[x^\frac{m_2+2}{2}-(2x_0)^\frac{m_2+2}{2}\right] \bigg\},
\end{align}
\begin{align}\label{BC2}
BC'=\bigg\{ (x,y)\in\mathbb{R}^2; \ y_{c'}\leq y\leq0, \ \frac{2}{m_1+2}(-y)^\frac{m_1+2}{2}=\frac{2}{m_2+2}x^\frac{m_2+2}{2} \bigg\},
\end{align}
and the point $C'=\left((\frac{1}{2})^\frac{2}{m_2+2}(2x_0),y_{c'}\right)$ where $$y_{c'}=-\left[\frac{1}{2}\left(\frac{m_1+2}{m_2+2}\right)(2x_0)^\frac{m_2+2}{2}\right]^\frac{2}{m_1+2}.$$

Parameterizing \eqref{AC2} and \eqref{BC2} in variable $y$,  the exterior normal vector be given by
\begin{align}\label{nAC2}
\eta_{AC'}=\left(1,-\left[-\frac{m_2+2}{m_1+2}(-y)^\frac{m_1+2}{2}+(2x_0)^\frac{m_2+2}{2}\right]^\frac{-m_2}{m_2+2}(-y)^\frac{m_1}{2}\right),
\end{align}
and
\begin{align}\label{nBC2}
\eta_{BC'}=\left(-1,-\left[\frac{m_2+2}{m_1+2}(-y)^\frac{m_1+2}{2}\right]^\frac{-m_2}{m_2+2}(-y)^\frac{m_1}{2}\right),
\end{align}
respectively. 


We also study the problem

\begin{equation}\label{P3}
\begin{cases}
  \mathcal{O}(u):=-y^{m_1}u_{xx}-x^{m_2}u_{yy}=f(u) &\mbox{in}\quad \Omega_3, \\
\quad \ \ \;u=0 &\mbox{on}\quad AC\cup\sigma_3\subseteq\partial\Omega_3,
\end{cases}
\end{equation}
where $\Omega_3$ is a Tricomi domain for the operator $\mathcal{O}=-y^{m_1}\partial^2_x-x^{m_2}\partial^2_y$. That is, $\Omega_3$ is open, bounded, simply connected set with piecewise $C^1$ boundary $\partial\Omega_3=\sigma_3\cup AC''\cup BC''$ formed by an elliptical arc $\sigma_3\subset\{(x,y)\in\mathbb{R}^2; \ x<0\}$ intercepting the axis $y$ at the points $A''=(0,2y_0)$ and $B''=(0,0)$ with $y_0<0$ and the characteristics curves $AC''$ and $BC''$ in $\{(x,y)\in\mathbb{R}^2; \ x\geq0 \ \mbox{and} \ y\leq0\}$ of $\mathcal{O}$ passing through points $A''$ and $B''$ respectively and meet at point $C''$. See figure \ref{Sorv3}.

\begin{figure}[!ht]
    \centering
    \includegraphics[width=8cm,angle=-90]{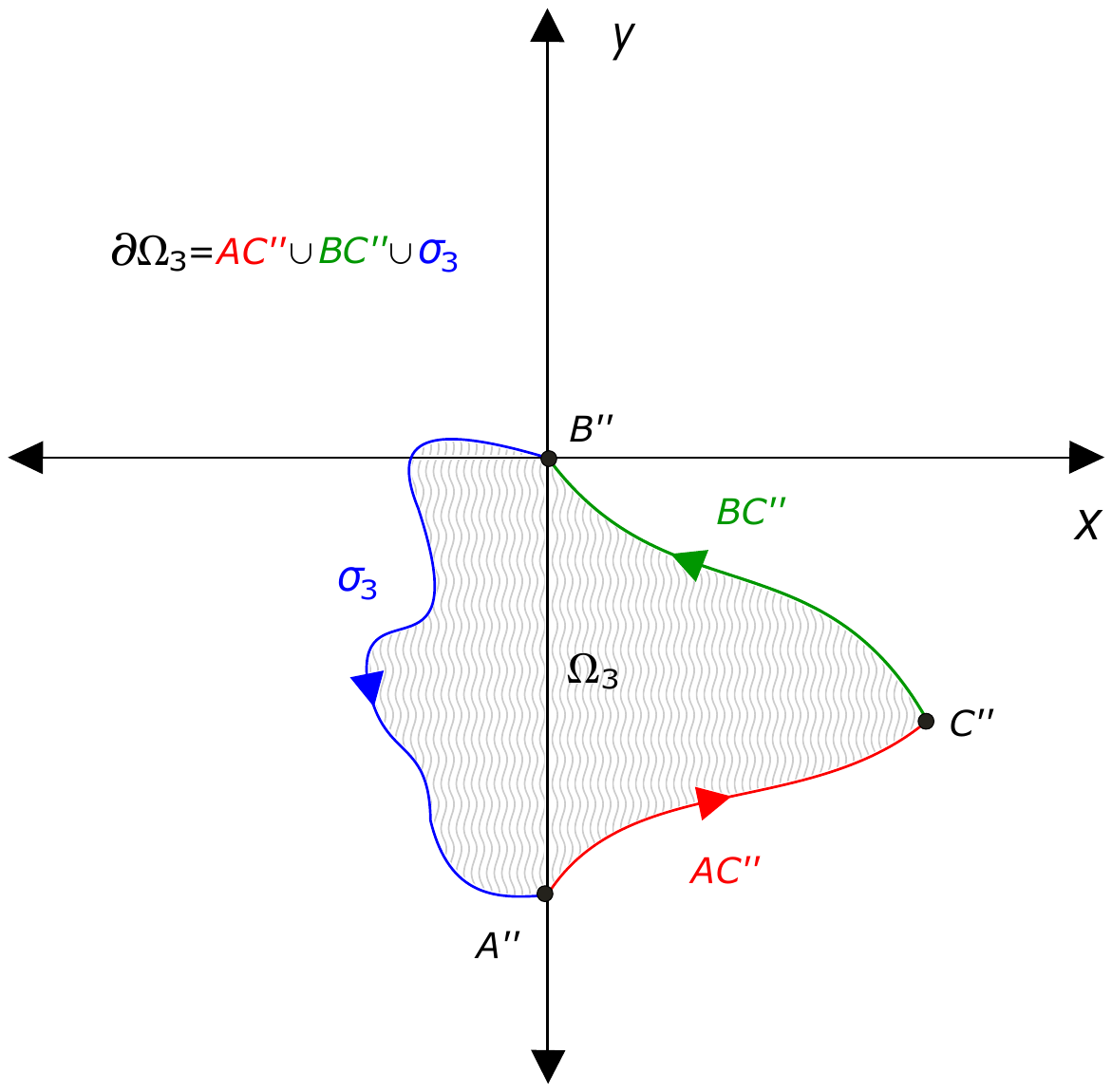}
    \caption{Domain $\Omega_3$ for the problem \eqref{P3}.}
    \label{Sorv3}
\end{figure}

To construction of the domain $\Omega_3$, we suppose $\Gamma(y)=(x(y),y)$, locally as the graph of a function, in this case $\alpha(y)=x(y)$ and $\beta(y)=y$, so $\alpha'(y)=x'(y)$ and $\beta'(y)=1$. Inserting in \eqref{CCC},  we concluded that the characteristics curves of the problem \eqref{P3} are solutions of 
$$-y^{m_1}=x^{m_2}(x'(y))^2.$$

With calculations analogous to those made to find the expressions of the characteristic curves \eqref{AC1} and \eqref{BC1} of problem \eqref{P1}, analyzing $|x'(y)|$ instead of $|y'(x)|$, the characteristics curves $AC''$ and $BC''$ in $\{(x,y)\in\mathbb{R}^2; \ x\geq0 \ \mbox{and} \ y\leq0\}$ of the operator $\mathcal{O}$, for $m_1\in\mathbb{N}$ odd number and $m_2\in\mathbb{N}$, are given by 
\begin{align}\label{AC3}
    AC''=\bigg\{ (x,y)&\in\mathbb{R}^2; \ 0\leq x\leq x_c, \nonumber\\
&-\frac{2}{m_1+2}\left[(-y)^\frac{m_1+2}{2}-(-2y_0)^\frac{m_1+2}{2}\right]=\frac{2}{m_2+2}x^\frac{m_2+2}{2} \bigg\},
\end{align}
\begin{align}\label{BC3}
BC''=\bigg\{ (x,y)\in\mathbb{R}^2; \ 0\leq x\leq x_{c''}, \ \frac{2}{m_1+2}(-y)^\frac{m_1+2}{2}=\frac{2}{m_2+2}x^\frac{m_2+2}{2} \bigg\},
\end{align}
and the point $C''=\left(x_{c''},(\frac{1}{2})^\frac{2}{m_1+2}(2y_0)\right)$ where $$x_{c''}=\left[\frac{1}{2}\left(\frac{m_2+2}{m_1+2}\right)(-2y_0)^\frac{m_1+2}{2}\right]^\frac{2}{m_2+2}.$$

Parameterizing \eqref{AC3} and \eqref{BC3} in variable $x$,  the exterior normal vector be given by
\begin{align}\label{nAC3}
\eta_{AC''}=\left( \left[-\frac{m_1+2}{m_2+2}x^\frac{m_2+2}{2}+(-2y_0)^\frac{m_1+2}{2}\right]^\frac{-m_1}{m_1+2}x^\frac{m_2}{2},-1 \right),
\end{align} 
and 
\begin{align}\label{nBC3}
\eta_{BC''}=\left(\left[\frac{m_1+2}{m_2+2}x^\frac{m_2+2}{2}\right]^\frac{-m_1}{m_1+2}x^\frac{m_2}{2},1\right),
\end{align}
respectively. 


Now, we will state the theorems with the Poho\^{z}aev-type identities for the mentioned problems \eqref{P1}, \eqref{P2} and \eqref{P3} and we will give the respective proof of the first one and we will comment on what is relevant to carry out the proof of the remaining two.

\begin{theorem}\label{T1}
Let $\Omega_1$ be a Tricomi domain for the operator $\mathcal{O}$ with $m_1\in\mathbb{N}$ odd and $\frac{m_2}{2}\in\mathbb{N}$ even numbers, whose boundary $\partial\Omega_1$ is piecewise $C^1$ with the exterior unit normal vector $\eta$. If $u\in C^2(\overline{\Omega_1})$ is a solution of the problem \eqref{P1}, then 
\begin{align}\label{PI1}
    (m_1+m_2+4)\int_{\Omega_1} F(u)-\frac{m_1+m_2+m_1m_2}{2}&\int_{\Omega_1} uf(u)= \nonumber\\ &\frac{1}{2}\left[\int_{BC\cup\sigma_1}\omega_1+\int_{BC}\omega_2\right],
\end{align}
where $F$ is the primitive function of $f\in C^0(\mathbb{R})$ such that $F(0)=0$, $\omega_1$ and $\omega_2$ are defined by 
\begin{align}\label{w1}
    \omega_1=\Bigl[2Du(-y^{m_1}u_x,&-x^{m_2}u_y) \nonumber\\ +(y^{m_1}u^2_x&+x^{m_2}u^2_y)(-(m_1+2)x,-(m_2+2)y\Bigl]\cdot\eta \, ,
\end{align}
\begin{align}\label{w2}
    \omega_2=\Bigl[-2F(u)(-(m_1+2)x,&-(m_2+2)y) \nonumber\\ 
    -(m_1+m_2&+m_1m_2)u (-y^{m_1}u_x,-x^{m_2}u_y)\Bigl]\cdot\eta \, ,
\end{align}
and $D$ is the vector field 
\begin{align}
    D=-(m_1+2)x\partial_x-(m_2+2)y\partial_y.
\end{align}
\end{theorem}

\begin{proof}
The proof will be done in 4 steps. In step 1, we will estimate $\int_{\Omega_1} Du\mathcal{O}u$. While in step 2, we will estimate $\int_{\Omega_1} Du f(u)$ and in step 3, we will estimate $\int_{\Omega_1} u f(u)$. Step 4, we will multiply Problem \eqref{P1} by $Du$, and we integrate over the domain where the regularity of the solution allows us to use the Divergence Theorem \cite{giltru} and we relate the previous steps to prove \eqref{PI1}.

\vspace{0,3cm}

\textbf{Step 1.} Claim,
\begin{align}\label{DOF}
    \int_{\Omega_1} Du \, \mathcal{O}u=&\frac{m_1+m_2+m_1m_2}{2} \int_{\Omega_1} \bigl(y^{m_1}u^2_x+x^{m_2}u^2_y\bigl) \nonumber\\ &+\frac{1}{2}\int_{BC\cup\sigma_1}\Bigl[2Du\bigl(-y^{m_1}u_x,-x^{m_2}u_y\bigr) \nonumber\\ &+\bigl(y^{m_1}u^2_x+x^{m_2}u^2_y\bigl)(-(m_1+2)x,-(m_2+2)y)\Bigl]\cdot\eta \, .
\end{align}

In fact, 
\begin{align}\label{div}
    div\Big((y^{m_1}u^2_x+x^{m_2}u^2_y)&(-(m_1+2)x,-(m_2+2)y)\Big) \nonumber \\ =&\nabla(y^{m_1}u^2_x+x^{m_2}u^2_y)(-(m_1+2)x,-(m_2+2)y)\nonumber \\
    &+(y^{m_1}u^2_x+x^{m_2}u^2_y)\, div\Big((-(m_1+2)x,-(m_2+2)y)\Big) \nonumber\\
    =&-2(m_1+2)xy^{m_1}u_xu_{xx}-2(m_1+2)x^{m_2+1}u_yu_{xy}\nonumber\\
    &-m_2(m_1+2)x^{m_2}u^2_y-m_1(m_2+2)y^{m_1}u^2_x-2(m_2+2)y^{m_1+1}u_xu_{xy}\nonumber\\
    &-2(m_2+2)yx^{m_1}u_yu_{yy}-(m_1+m_2+4)(y^{m_1}u^2_x+x^{m_2}u^2_y),
\end{align}

Define $\mathcal{X}u=(-y^{m_1}u_x,-x^{m_2}u_y)$. Since $Du=-(m_1+2)x\,u_x-(m_2+2)y\,u_y$, we found by simple calculations \begin{align}\label{DX}
    -\nabla(Du)\cdot\mathcal{X}u=&-(m_1+2)y^{m_1}u^2_x-(m_1+2)xy^{m_1}u_xu_{xx}\nonumber\\ 
    &-(m_2+2)y^{m_1+1}u_xu_{xy}-(m_1+2)x^{m_2+1}u_yu_{xy}\nonumber\\ 
    &-(m_2+2)x^{m_2}u^2_y-(m_2+2)yx^{m_2}u_yu_{yy}.
\end{align}
So, from \eqref{div} and \eqref{DX}, we get
\begin{align}\label{divDX}
    -\nabla(Du)\cdot\mathcal{X}u=&\frac{1}{2} \, div\left((y^{m_1}u^2_x+x^{m_2}u^2_y)(-(m_1+2)x,-(m_2+2)y)\right)\nonumber\\
    &+\frac{(m_1+m_2+m_1m_2)}{2}\bigl(y^{m_1}u^2_x+x^{m_2}u^2_y\bigl).
\end{align}

Using the Divergence Theorem \cite{giltru}, we gain

$$\int_{\Omega_1} div(Du\,\mathcal{X}u)=\int_{\partial\Omega_1}Du\,\mathcal{X}u\cdot\eta \, ,$$
by the proprieties of the divergent 

$$\int_{\Omega_1}\nabla(Du)\cdot\mathcal{X}u+\int_{\Omega_1} Du \, div(\mathcal{X}u)=\int_{\partial\Omega_1}Du\,\mathcal{X}u\cdot\eta \, .$$

Note that, $div(\mathcal{X}u)=\mathcal{O}u$. So, we obtain 

$$\int_{\Omega_1} Du \, \mathcal{O}u=\int_{\partial\Omega_1}Du\,\mathcal{X}u\cdot\eta+\int_{\Omega_1}-\nabla(Du)\cdot\mathcal{X}u.$$

Inserting \eqref{divDX} in the last equation 
\begin{align}\label{TDdivDX}
    \int_{\Omega_1} Du \, \mathcal{O}u=&\int_{\partial\Omega_1}Du\,\mathcal{X}u\cdot\eta \nonumber\\ 
    &+\frac{1}{2}\int_\Omega div\left((y^{m_1}u^2_x+x^{m_2}u^2_y)(-(m_1+2)x,-(m_2+2)y)\right) \nonumber\\ 
    &+\frac{(m_1+m_2+m_1m_2)}{2}\int_\Omega\bigl(y^{m_1}u^2_x+x^{m_2}u^2_y\bigl).
\end{align}

Using the Divergence Theorem again, we have
\begin{align}\label{TDdiv}
    \int_{\Omega_1} div&\left((y^{m_1}u^2_x+x^{m_2}u^2_y)(-(m_1+2)x,-(m_2+2)y)\right)= \nonumber\\ &\int_{\partial\Omega_1}\bigl(y^{m_1}u^2_x+x^{m_2}u^2_y\bigl)(-(m_1+2)x,-(m_2+2)y)\cdot\eta \, ,
\end{align}
and inserting \eqref{TDdiv} in \eqref{TDdivDX} we obtain
\begin{align*}
    \int_{\Omega_1} Du \, \mathcal{O}u=&\frac{(m_1+m_2+m_1m_2)}{2}\int_{\Omega_1}(y^{m_1}u^2_x+x^{m_2}u^2_y)\\ &+\frac{1}{2}\int_{\partial\Omega_1}\bigl[2Du(-y^{m_1}u_x,-x^{m_2}u_y) \\
    &+\bigl(y^{m_1}u^2_x+x^{m_2}u^2_y\bigl)(-(m_1+2)x,-(m_2+2)y)\bigl]\cdot\eta  \, ,
\end{align*}
this last equation is the identity \eqref{DOF} if the boundary integral is vanished on $AC$. Just prove that, $(y^{m_1}u^2_x+x^{m_2}u^2_y)_{|_{AC}}\equiv0$ and $(\mathcal{X}u\cdot\eta)_{|_{AC}}\equiv0$. Denote by, 
$$\partial_+u:=x^{-\frac{m_2}{2}}\Big[x^{\frac{m_2}{2}}u_y+(-y)^{\frac{m_1}{2}}u_x\Big],$$
and
$$\partial_-u:=x^{-\frac{m_2}{2}}\Big[x^{\frac{m_2}{2}}u_y-(-y)^{\frac{m_1}{2}}u_x\Big].$$

So, \begin{align*}
    (\partial_+u)(\partial_-u)&=x^{-\frac{m_2}{2}}\Big[x^{\frac{m_2}{2}}u_y+(-y)^{\frac{m_1}{2}}u_x\Big]x^{-\frac{m_2}{2}}\Big[x^{\frac{m_2}{2}}u_y-(-y)^{\frac{m_1}{2}}u_x\Big] \\
    &=x^{-m_2}\Big[x^{m_2}u_y^2-(-y)^{m_1}u_x^2\Big] \\
    &=x^{-m_2}\Big[x^{m_2}u_y^2+y^{m_1}u_x^2\Big].
\end{align*}

Moreover, we have  
\begin{align*}
    (-y^{m_1}&u_x,-x^{m_2}u_y)\cdot\eta_{AC}\\&=(y^{m_1}u^2_x+x^{m_2}u^2_y)\cdot\left(-1,-\left[\frac{m_2+2}{m_1+2}(-y)^\frac{m_1+2}{2}+(2x_0)^\frac{m_2+2}{2}\right]^\frac{-m_2}{m_2+2}(-y)^\frac{m_1}{2}\right) \\
    &=y^{m_1}u_x+x^{m_2}(-y)^{\frac{m_1}{2}}\Big[x^{\frac{m_2+2}{2}}\Big]^{-\frac{m_2}{m_2+2}}u_y \\
    &=(-y)^{\frac{m_1}{2}}x^{\frac{m_2}{2}}x^{-\frac{m_2}{2}}\Big[x^{\frac{m_2}{2}}u_y-(-y)^{\frac{m_1}{2}}u_x\Big] \\
    &=(-y)^{\frac{m_1}{2}}x^{\frac{m_2}{2}}\partial_-u.
\end{align*}

See that, $\partial_+u$ and $\partial_-u$ are essentially the directional derivatives along the characteristics $BC$ and $AC$ respectively. By hypothesis $u_{|_{AC}}\equiv 0$ which implies $\partial_-u_{|_{AC}}\equiv 0$, thus $(y^{m_1}u^2_x+x^{m_2}u^2_y)_{|_{AC}}\equiv0$ and $(\mathcal{X}u\cdot\eta)_{|_{AC}}\equiv0$. In conclusion, for $m_1$, $\frac{m_2}{2}$ odd and even numbers, respectively, \eqref{DOF} is true.

\vspace{0,3cm}
\textbf{Step 2.} Claim, 
\begin{align}\label{Df}
    \int_{\Omega_1} Du\,f(u)=&(m_1+m_2+4)\int_{\Omega_1} F(u) \nonumber\\
    &+\int_{BC}F(u)(-(m_1+2)x,-(m_2+2)y)\cdot\eta \, .
\end{align}

In fact. By the Divergence Theorem, we obtain
\begin{align*}
 \int_{\Omega_1} div\bigl( F(u)(-(m_1+2)x,&-(m_2+2)y) \bigl)= \nonumber\\
 &\int_{\partial\Omega_1}F(u)(-(m_1+2)x,-(m_2+2)y)\cdot\eta \, ,
\end{align*}
by the divergent´s proprieties gives
\begin{align}
\int_{\Omega_1}\nabla F(u)\cdot(-(&m_1+2)x,-(m_2+2)y)\nonumber\\
    +\int_{\Omega_1} F(u) \, div\bigl( (-&(m_1+2)x,-(m_2+2)y) \bigl)=\nonumber\\ &\int_{\partial\Omega_1}F(u)(-(m_1+2)x,-(m_2+2)y)\cdot\eta \, .
\end{align}

By the hypotheses, $F$ is the primitive function of $f$ so, $\nabla F(u)=f(u)\nabla u$. Also, $div\bigl( (-(m_1+2)x,-(m_2+2)y) \bigl)=-m_1-m_2-4$. Thus, 
\begin{align*}
    \int_{\Omega_1} f(u)&\nabla u\cdot(-(m_1+2)x,-(m_2+2)y)= \nonumber\\ &\int_{\partial\Omega_1}F(u)(-(m_1+2)x,-(m_2+2)y)\cdot\eta+(m_1+m_2+4)\int_{\Omega_1} F(u) \, .
\end{align*}

Note that, $$\nabla u\cdot(-(m_1+2)x,-(m_2+2)y)=-(m_1+2)x\,u_x-(m_2+2)y\,u_y=Du.$$ Then,
\begin{align*}
    \int_{\Omega_1} Du\,f(u)=&(m_1+m_2+4)\int_{\Omega_1} F(u) \nonumber\\
    &+\int_{\partial\Omega_1}F(u)(-(m_1+2)x,-(m_2+2)y)\cdot\eta\, .
\end{align*}
To prove \eqref{Df} just see that $u_{|_{AC\cup\sigma_1}}=0$ by the boundary condition and $F(0)=0$ by hypotheses.

\vspace{0,3cm}

\textbf{Step 3.} Claim, 
\begin{align}\label{uf}
    \int_{\Omega_1} u\,f(u)=\int_{\Omega_1}(y^{m_1}u^2_x+x^{m_2}u^2_y)+\int_{BC}u\,(-y^{m_1}u_x,-x^{m_2}u_y)\cdot\eta \, .
\end{align}

In fact. By the Divergence Theorem, we get 
\begin{align*}
    \int_{\Omega_1} div\left(u\,(-y^{m_1}u_x,-x^{m_2}u_y)\right)=\int_{\partial\Omega_1}u\,(-y^{m_1}u_x,-x^{m_2}u_y)\cdot\eta\, ,
\end{align*}
by the divergent´s proprieties give us
\begin{align*}
    \int_{\Omega_1} \nabla u\cdot(-y^{m_1}u_x,-x^{m_2}u_y) + \int_{\Omega_1} u\,div\bigl((-y^{m_1}&u_x,-x^{m_2}u_y)\bigl)= \\ &\int_{\partial\Omega_1}u\,(-y^{m_1}u_x,-x^{m_2}u_y)\cdot\eta\, .
\end{align*}

Since $\nabla u =(u_x,u_y)$ and $div\bigl((-y^{m_1}u_x,-x^{m_2}u_y)\bigl)=\mathcal{O}u$, we get
\begin{align*}
    \int_{\Omega_1} u\,\mathcal{O}u=\int_{\Omega_1}(y^{m_1}u^2_x+x^{m_2}u^2_y)+ \int_{\partial\Omega_1}u\,(-y^{m_1}u_x,-x^{m_2}u_y)\cdot\eta\, .
\end{align*}
This last equation is \eqref{uf} just observing that, $\mathcal{O}u=f(u)$ in $\Omega_1$ and $u_{|_{AC\cup\sigma_1}}=0$.

\vspace{0,3cm}

\textbf{Step 4.} Proof the Poho\^{z}aev identity \eqref{PI1}. 

\vspace{0,3cm}
Multiplying $\mathcal{O}u=f(u)$ by $Du$ in $\Omega_1$, inserting \eqref{DOF} from Step 1 and \eqref{Df} from Step 2, gives
\begin{align}\label{DODf}
   (m_1+m_2+4)\int_{\Omega_1} F(u)=& \frac{m_1+m_2+m_1m_2}{2} \int_{\Omega_1} \bigl(y^{m_1}u^2_x+x^{m_2}u^2_y\bigl) \nonumber\\ &+\frac{1}{2}\int_{BC\cup\sigma_1}\Bigl[2Du\bigl(-y^{m_1}u_x,-x^{m_2}u_y\bigr) \nonumber\\ &+\bigl(y^{m_1}u^2_x+x^{m_2}u^2_y\bigl)(-(m_1+2)x,-(m_2+2)y)\Bigl]\cdot\eta \nonumber\\ &-\int_{BC}F(u)(-(m_1+2)x,-(m_2+2)y)\cdot\eta\, .
\end{align}

Multiplying \eqref{uf} in Step 3 by $-\frac{m_1+m_2+m_1m_2}{2}$ give us
\begin{align}\label{ufM}
    -\frac{m_1+m_2+m_1m_2}{2}&\int_{\Omega_1} u\,f(u)=-\frac{m_1+m_2+m_1m_2}{2}\int_{\Omega_1}(y^{m_1}u^2_x+x^{m_2}u^2_y)\nonumber\\ &-\frac{m_1+m_2+m_1m_2}{2}\int_{BC}u\,(-y^{m_1}u_x,-x^{m_2}u_y)\cdot\eta \, .
\end{align}

Adding the identities \eqref{DODf} and \eqref{ufM}, grouping the integrals and inserting \eqref{w1} and \eqref{w2}, we get the result. 
\end{proof}

\begin{theorem}\label{T2}
Let $\Omega_2$ be a Tricomi domain for the operator $\mathcal{O}$ with $m_1\in\mathbb{N}$ odd number and $m_2\in\mathbb{N}$, whose boundary $\partial\Omega_2$ is piecewise $C^1$ with the exterior unitary normal vector $\eta$. If $u\in C^2(\overline{\Omega_2})$ is a solution of the problem \eqref{P2}, then 
\begin{align}\label{PI2}
    (m_1+m_2+4)\int_{\Omega_2} F(u)-\frac{m_1+m_2+m_1m_2}{2}&\int_{\Omega_2} uf(u)= \nonumber\\ &\frac{1}{2}\left[\int_{BC'\cup\sigma_2}\omega_1+\int_{BC'}\omega_2\right],
\end{align}
where $F$ is the primitive of function of $f\in C^0(\mathbb{R})$ such that $F(0)=0$, $\omega_1$ and $\omega_2$ as in \eqref{w1} and \eqref{w2} respectively.
\end{theorem}

\begin{theorem}\label{T3}
Let $\Omega_3$ be a Tricomi domain for the operator $\mathcal{O}$ with $m_1\in\mathbb{N}$ odd number and $m_2\in\mathbb{N}$, whose boundary $\partial\Omega_3$ is piecewise $C^1$ with the exterior unitary normal vector $\eta$. If $u\in C^2(\overline{\Omega_3})$ is a solution of the problem \eqref{P3}, then 
\begin{align}\label{PI3}
    (m_1+m_2+4)\int_{\Omega_3} F(u)-\frac{m_1+m_2+m_1m_2}{2}&\int_{\Omega_3} uf(u)= \nonumber\\ &\frac{1}{2}\left[\int_{BC''\cup\sigma_3}\omega_1+\int_{BC''}\omega_2\right],
\end{align}
where $F$ is the primitive function of $f\in C^0(\mathbb{R})$ such that $F(0)=0$, $\omega_1$ and $\omega_2$ as in \eqref{w1} and \eqref{w2} respectively.
\end{theorem}

The Proof of Theorems \ref{T2} and \ref{T3} are similar to the proof of the Theorem \ref{T1}. Just observe that, the Poho\^{z}aev-type identities are similar given in terms of $BC$, $BC'$ and $BC''$. Note that, $\mathcal{X}u$, the operator $\mathcal{O}u$ and the vector field $Du$ do not change to the problems \eqref{P1}, \eqref{P2} and \eqref{P3}, so $\omega_1$ and $\omega_2$ have the same expression independently of the problem. Moreover, the computations done in the Step 2, Step 3 and Step 4 of the Proof of Theorem \ref{T1} which remain valid  to the proof of Theorems \ref{T2} and \ref{T3}. But, the Step 1 of the proof of Theorem \ref{T1} depends strongly on the facts $(y^{m_1}u^2_x+x^{m_2}u^2_y)_{|_{AC}}\equiv0$ and $(\mathcal{X}u\cdot\eta)_{|_{AC}}\equiv0$. The proof of Theorems \ref{T2} and \ref{T3}, we need to define the directional derivatives of $u$ along the characteristic curves $AC'$, $BC'$, $AC''$ and $BC''$ and with these, we found $(y^{m_1}u^2_x+x^{m_2}u^2_y)_{|_{AC'}}\equiv0$, $(y^{m_1}u^2_x+x^{m_2}u^2_y)_{|_{AC''}}\equiv0$ and $(\mathcal{X}u\cdot\eta)_{|_{AC'}}\equiv0$, $(\mathcal{X}u\cdot\eta)_{|_{AC''}}\equiv0$ as done in the Step 1 of the proof of Theorem \ref{T1}. Therefore, we can affirmed that the Poho\^{z}aev-type identities are true to the proposed problems.


\section{Nonexistence results}\label{Sec4}

Together with the Poho\^{z}aev-type equations \eqref{PI1},  \eqref{PI2} and  \eqref{PI3} from Section \ref{Sec3}, the nonexistence results presented in this Section depend directly on the Hardy-Sobolev inequality because it controls the sign of the boundary integrals on the characteristic curves of $\Omega_1$, $\Omega_2$  and $\Omega_3$. The relationship between the critical growth of power type in nonlinearity and the critical exponent of the weighted Sobolev embedding in the Theorem \ref{IMER} is explicit in each of these nonexistence results. We define a class of weighted functions and a class of absolutely continuous functions where the Hardy-Sobolev inequality will be applied. This inequality can be seen as a weighted Sobolev inequality and the proof can be found in [\cite{opickufner}, Theorem 1.14]. Let,
$$\mathcal{W}=\left\{ w: [a,b]\rightarrow\mathbb{R}; w \, \, \mbox{is measurable, positive and finite a.e. in} \, \, [a,b]\right\}$$
and $\mathcal{AC}_L(a,b)$ defined by 
$$\left\{ \phi:[a,b]\rightarrow\mathbb{R} \,;\, \phi_{|J}\in\mathcal{AC}(J), \, \forall\, J=[c,d]\subset(a,b) \, \mbox{and} \lim_{x\rightarrow a^+}\phi(x)=0 \right\}.$$

\begin{lemma}\label{HSI}
(\textit{Hardy-Sobolev Inequality}). Let $1<p\leq q<+\infty$ and $v,w\in\mathcal{W}$ be given. Then
$$\left[\int_a^b|\phi(x)|^qw(x)dx \right]^\frac{1}{q}\leq C_L\left[\int_a^b|\phi'(x)|^pv(x)dx \right]^\frac{1}{p},$$
for all $\phi\in\mathcal{AC}_L(a,b)$, if and only if, $$M_L=\sup_{a<x<b}G_L(x)<+\infty\,,$$
where 
$$G_L(x):=\left(\int_x^bw(t)dt\right)^\frac{1}{q}\left(\int_a^xv^{1-p'}(t)dt\right)^\frac{1}{p'}.$$

Moreover, the best constant $C_L$ holds
$$M_L\leq C_L\leq r(p,q)M_L \, \, \, \mbox{where} \,\,\, r(p,q):=\left(1+\frac{q}{p'}\right)^\frac{1}{q}\left(1+\frac{p'}{q}\right)^\frac{1}{q'}.$$
\end{lemma}

\begin{theorem}\label{T4}
Let $\Omega_1$ be a Tricomi domain for the operator $\mathcal{O}$, which is $D$-star-shaped where $D=-(m_1+2)x\partial_x-(m_2+2)y\partial_y$ with $m_1$ odd and $\frac{m_2}{2}$ even numbers. Let $u\in C^2(\overline{\Omega_1})$ be a regular solution of 

\begin{equation}\label{P1+}
 \begin{cases}
\mathcal{O}(u):=-y^{m_1}u_{xx}-x^{m_2}u_{yy}=u|u|^{\alpha-1} &\mbox{in} \ \ \Omega_1, \\
\quad \ \ \;u=0 &\mbox{on} \ \ AC\cup\sigma_1\subseteq\partial\Omega_1,
\end{cases}
\end{equation}
with $\alpha>2^*(m_1,m_2)-1=\frac{m_1+m_2-m_1m_2+8}{m_1+m_2+m_1m_2}$. Then $u\equiv0$ a.e. in $\Omega_1$.
\end{theorem}

\begin{proof}
The proof will be done in 5 steps. Step 1, we will analyzed $f(u)$ and $F(u)$. Step 2, we apply the Poho\^{z}aev-type equation (Theorem \ref{T1}). Step 3, we will check that $\int_{\sigma_1}\omega_1\geq0$. Step 4, we will rewrite $\int_{BC}(\omega_1+\omega_2)$ by a parameterization. Finally, Step 5, we apply Lemma \ref{HSI} to prove that  $\int_{BC}(\omega_1+\omega_2)\geq0$.

\vspace{0,3cm}

\textbf{Step 1.} The problem \eqref{P1+} is a particular case of the problem \eqref{P1} where $f(u)=u|u|^{\alpha-1}$. Which implies, $F(u)=\frac{|u|^{\alpha+1}}{\alpha+1}$.

\vspace{0,3cm}

\textbf{Step 2.} Applying the Poho\^{z}aev-type equation, Theorem \ref{T1}, we get

\begin{align*}
    (m_1+m_2+4)\int_{\Omega_1} \frac{|u|^{\alpha+1}}{\alpha+1}-\frac{m_1+m_2+m_1m_2}{2}&\int_{\Omega_1} |u|^{\alpha+1}= \nonumber\\ &\frac{1}{2}\left[\int_{BC\cup\sigma_1}\omega_1+\int_{BC}\omega_2\right].
\end{align*}

So, we obtain
\begin{align*}
 \left(\frac{(m_1+m_2+4)}{\alpha+1}-\frac{m_1+m_2+m_1m_2}{2}\right)\int_{\Omega_1} |u|^{\alpha+1}= \frac{1}{2}\left[\int_{BC\cup\sigma_1}\omega_1+\int_{BC}\omega_2\right],
\end{align*}
 and 
 
\begin{align*}
 \left(\frac{m_1+m_2-m_1m_2+8-\alpha(m_1+m_2+m_1m_2)}{2(\alpha+1)}\right)&\int_{\Omega_1} |u|^{\alpha+1}=\\ &\frac{1}{2}\left[\int_{BC\cup\sigma_1}\omega_1+\int_{BC}\omega_2\right].
\end{align*}
 
 In the next Steps, we will prove that 
$$\frac{1}{2}\left[\int_{BC\cup\sigma_1}\omega_1+\int_{BC}\omega_2\right]\geq0.$$
Since $|u|^{\alpha+1}\geq 0$, then follow the result.

\vspace{0,3cm}
\textbf{Step 3.} By \eqref{w1}, we obtain 
\begin{align*}
    \int_{\sigma_1}\omega_1ds=&\int_{\sigma_1}\Bigl[2Du(-y^{m_1}u_x,-x^{m_2}u_y) \nonumber\\ &+(y^{m_1}u^2_x+x^{m_2}u^2_y)(-(m_1+2)x,-(m_2+2)y)\Bigl]\cdot\eta_{\sigma_1} ds \, ,
\end{align*}
since $Du=-(m_1+2)x\,u_x-(m_2+2)y\,u_y$ and $\eta_{\sigma_1} ds=(-dy,dx)$, then
\begin{align*}
    \int_{\sigma_1}\omega_1ds=&\int_{\sigma_1}\Bigl[(-2(m_1+2)x\,u_x-2(m_2+2)y\,u_y)(-y^{m_1}u_x)\\
    &+(y^{m_1}u^2_x+x^{m_2}u^2_y)(-(m_1+2)x) \Bigl](-dy) \\ &+\int_{\sigma_1}\Bigl[(-2(m_1+2)x\,u_x-2(m_2+2)y\,u_y)(-x^{m_2}u_y)\\
    &+(y^{m_1}u^2_x+x^{m_2}u^2_y)(-(m_2+2)y)\Bigl]dx \, .
\end{align*}

 Thus, we get 
\begin{align*}
    \int_{\sigma_1}\omega_1ds=&\int_{\sigma_1}(y^{m_1}u^2_x+x^{m_2}u^2_y)((m_1+2)xdy-(m_2+2)ydx)\\
    &+\int_{\sigma_1}(2(m_1+2)xx^{m_2}\,u_y-2(m_2+2)yy^{m_1}\,u_x)(u_xdx+u_ydy).
\end{align*}

By hypotheses $u_{|\sigma_1}\equiv0$ then $(u_xdx+u_ydy)_{|\sigma_1}=0$. So, we have
\begin{align*}
    \int_{\sigma_1}\omega_1ds=\int_{\sigma_1}(y^{m_1}u^2_x+x^{m_2}u^2_y)((m_1+2)xdy-(m_2+2)ydx).
\end{align*}

We know that $\Omega_1$ is $D$-star-shaped and $\partial\Omega_1$ is $D$-starlike in the sense of Definition \ref{OPIE}. Therefore, $((m_1+2)xdy-(m_2+2)ydx)\geq0$. See well, $(y^{m_1}u^2_x+x^{m_2}u^2_y)_{|\sigma_1}\geq0$ because $\sigma_1\subset\{(x,y)\in\mathbb{R}^2 ;\, y>0 \}$ and $m_2$ is even. In conclusion,
\begin{align*}
    \int_{\sigma_1}\omega_1ds\geq0.
\end{align*}

\textbf{Step 4.} Denote by 
$$I=\int_{BC}(\omega_1+\omega_2).$$

We know that $\omega_1$ and $\omega_2$ are defined by \eqref{w1} and \eqref{w2} respectively.  Therefore, the generated flow by $D=-(m_1+2)x\partial_x-(m_2+2)y\partial_y$ is everywhere tangential to $BC$. Then, we obtain  
\begin{align}\label{I}
    I=\int_{BC}[2Du-(m_1+m_2&+m_1m_2)u](-y^{m_1}u_x,-x^{m_2}u_y)\cdot\eta\, ds \, .
\end{align}

Remembering, $BC$ is given by $-\frac{m_2+2}{m_1+2}(-y)^\frac{m_1+2}{2}=x^\frac{m_2+2}{2}$, we get on $BC$
\begin{align}\label{ds}
    ds=|\eta_{BC}|dy \ \ \mbox{and} \ \ Du=(m_2+2)(-y)\partial_+u ,
\end{align}
where $\partial_+u=u_y+x^{-\frac{m_2}{2}}(-y)^\frac{m_1}{2}u_x$ as defined in the Step 1 of the Theorem \ref{T1}.

\vspace{0.3cm}
Inserting \eqref{ds} and \eqref{nBC1} in \eqref{I}, we have
\begin{align*}
    I=\int^0_{y_c}[2(m_2+2)(-y)&\partial_+u-(m_1+m_2+m_1m_2)u]\\
    &(-y^{m_1}u_x,-x^{m_2}u_y)\cdot\frac{(1,-x^{-\frac{m_2}{2}}(-y)^\frac{m_1}{2})}{|\eta_{BC}|}|\eta_{BC}|dy \, ,
\end{align*}
so, 
\begin{align*}
    I=\int^0_{y_c}[2(m_2+2)(-y)\partial_+u-(m_1+m_2+m_1m_2)u](-y)^\frac{m_1}{2}x^\frac{m_2}{2}\partial_+u\, dy \, .
\end{align*}

Then, 
\begin{align}\label{I2}
    I=\int^0_{y_c}[2(m_2+2)&(-y)^\frac{m_1+2}{2}x^\frac{m_2}{2}(\partial_+u)^2  
    -(m_1+m_2+m_1m_2)(-y)^\frac{m_1}{2}x^\frac{m_2}{2}u \, \partial_+u] \, dy \, .
\end{align}

Parameterizing the curve $BC$ by 
\begin{align}\label{parame}
    \Gamma(t)=\left( \Bigl(-\frac{m_2+2}{m_1+2}(-t)^\frac{m_1+2}{2}\Bigl)^\frac{2}{m_2+2} , t\right) , \quad \quad t\in [y_c,0] \, ,
\end{align}
define $\phi(t)=u(\Gamma(t))\in C^2((y_c,0))\cup C^1([y_c,0])$ and note that 
\begin{align}\label{phi'}
    \phi'(t)=\nabla u\cdot\Gamma'(t)=\partial_+u(\Gamma(t)).
\end{align}

From \eqref{parame} and \eqref{phi'}, \eqref{I2} becomes
\begin{align}\label{Iphi'phi}
     I=&\int^0_{y_c}\Biggl[2(m_2+2)(-t)^\frac{m_1+2}{2}\Bigl(-\frac{m_2+2}{m_1+2}(-t)^\frac{m_1+2}{2}\Bigl)^\frac{m_2}{m_2+2}(\phi'(t))^2 \nonumber\\ 
    &-(m_1+m_2+m_1m_2)(-t)^\frac{m_1}{2}\Bigl(-\frac{m_2+2}{m_1+2}(-t)^\frac{m_1+2}{2}\Bigl)^\frac{m_2}{m_2+2}\phi(t)\phi'(t)\Biggl] \, dt \, .
\end{align}

See that, 
\begin{align*}
    (-t)^\frac{m_1+2}{2}\Bigl(-\frac{m_2+2}{m_1+2}(-t)^\frac{m_1+2}{2}\Bigl)^\frac{m_2}{m_2+2}=\Bigl(-\frac{m_2+2}{m_1+2}\Bigl)^\frac{m_2}{m_2+2}(-t)^\frac{m_1+2m_2+m_1m_2+2}{m_2+2},
\end{align*}
and
\begin{align*}
    (-t)^\frac{m_1}{2}\Bigl(-\frac{m_2+2}{m_1+2}(-t)^\frac{m_1+2}{2}\Bigl)^\frac{m_2}{m_2+2}=\Bigl(-\frac{m_2+2}{m_1+2}\Bigl)^\frac{m_2}{m_2+2}(-t)^\frac{m_1+m_2+m_1m_2}{m_2+2}.
\end{align*}
Furthermore,
\begin{align*}
    &\frac{d}{dt}\Biggl[\Bigl(-\frac{m_2+2}{m_1+2}\Bigl)^\frac{m_2}{m_2+2}(-t)^\frac{m_1+m_2+m_1m_2}{m_2+2}(\phi(t))^2\Biggl]=\\
    &\Bigl(-\frac{m_2+2}{m_1+2}\Bigl)^\frac{m_2}{m_2+2} \frac{d}{dt}\Biggl[(-t)^\frac{m_1+m_2+m_1m_2}{m_2+2}(\phi(t))^2\Biggl]=\\
    &-\Bigl(-\frac{m_2+2}{m_1+2}\Bigl)^\frac{m_2}{m_2+2}\Bigl(\frac{m_1+m_2+m_1m_2}{m_2+2}\Bigl)(-t)^\frac{m_1+m_1m_2-2}{m_2+2}(\phi(t))^2\\
    &+2\Bigl(-\frac{m_2+2}{m_1+2}\Bigl)^\frac{m_2}{m_2+2}(-t)^\frac{m_1+m_2+m_1m_2}{m_2+2}\phi(t)\phi'(t),
\end{align*}
then, 
\begin{align}\label{phiphi'phi2}
    &\Bigl(-\frac{m_2+2}{m_1+2}\Bigl)^\frac{m_2}{m_2+2}(-t)^\frac{m_1+m_2+m_1m_2}{m_2+2}\phi(t)\phi'(t)=\nonumber\\ 
    &\frac{1}{2}\frac{d}{dt}\Biggl[\Bigl(-\frac{m_2+2}{m_1+2}\Bigl)^\frac{m_2}{m_2+2}(-t)^\frac{m_1+m_2+m_1m_2}{m_2+2}(\phi(t))^2\Biggl]\nonumber\\
    &+\frac{1}{2}\Bigl(-\frac{m_2+2}{m_1+2}\Bigl)^\frac{m_2}{m_2+2}\Bigl(-\frac{m_1+m_2+m_1m_2}{m_2+2}\Bigl)(-t)^\frac{m_1+m_1m_2-2}{m_2+2}(\phi(t))^2.
\end{align}

Inserting \eqref{phiphi'phi2} in \eqref{Iphi'phi} and using the fact that $\phi(y_c)=0=\phi(0)$, we find 
\begin{align}\label{IF}
    I=&2(m_2+2)\Bigl(-\frac{m_2+2}{m_1+2}\Bigl)^\frac{m_2}{m_2+2}\int^0_{y_c}(-t)^\frac{m_1+2m_2+m_1m_2+2}{m_2+2}(\phi'(t))^2\, dt \nonumber\\ 
    &-\frac{1}{2}\frac{(m_1+m_2+m_1m_2)^2}{m_2+2}\Bigl(-\frac{m_2+2}{m_1+2}\Bigl)^\frac{m_2}{m_2+2}\int^0_{y_c}(-t)^\frac{m_1+m_1m_2-2}{m_2+2}(\phi(t))^2 \, dt \, .
\end{align}

\textbf{Step 5.} Claim. $I=\int_{BC}(\omega_1+\omega_2)\geq0$. This is,  the equation \eqref{IF} is nonnegative. Just see, 
\begin{align*}
    &\frac{1}{2}\frac{(m_1+m_2+m_1m_2)^2}{m_2+2}\Bigl(-\frac{m_2+2}{m_1+2}\Bigl)^\frac{m_2}{m_2+2}\int^0_{y_c}(-t)^\frac{m_1+m_1m_2-2}{m_2+2}(\phi(t))^2 \, dt \nonumber\\ 
    &\leq 2(m_2+2)\Bigl(-\frac{m_2+2}{m_1+2}\Bigl)^\frac{m_2}{m_2+2}\int^0_{y_c}(-t)^\frac{m_1+2m_2+m_1m_2+2}{m_2+2}(\phi'(t))^2\, dt \, ,
\end{align*}
simplifying,
\begin{align*}
    &\frac{1}{2}\frac{(m_1+m_2+m_1m_2)^2}{m_2+2}\int^0_{y_c}(-t)^\frac{m_1+m_1m_2-2}{m_2+2}(\phi(t))^2 \, dt \nonumber\\ 
    &\leq 2(m_2+2)\int^0_{y_c}(-t)^\frac{m_1+2m_2+m_1m_2+2}{m_2+2}(\phi'(t))^2\, dt \, ,
\end{align*}
equivalent to say, 
\begin{align}
    &\Biggl[\int^0_{y_c}(-t)^\frac{m_1+m_1m_2-2}{m_2+2}(\phi(t))^2 \, dt\Biggl]^\frac{1}{2} \nonumber\\ 
    &\leq \frac{2(m_2+2)}{m_1+m_2+m_1m_2}\Biggl[\int^0_{y_c}(-t)^\frac{m_1+2m_2+m_1m_2+2}{m_2+2}(\phi'(t))^2\, dt\Biggl]^\frac{1}{2} \, .
\end{align}

 For the function $\phi\in\mathcal{AC}_L(y_c,0)$, we use the Lemma \ref{HSI} the Hardy-Sobolev Inequality, with constant $C_L=\frac{2(m_2+2)}{m_1+m_2+m_1m_2}$ in the interval $(a,b)=(y_c,0)$, with exponents $p=q=p'=2$ and with weighted functions $v(t):=(-t)^\frac{m_1+2m_2+m_1m_2+2}{m_2+2}$ and $w(t)=(-t)^\frac{m_1+m_1m_2-2}{m_2+2}$.  Note that,
\begin{align*}
    G_L(x)=\Biggl[\int^0_x (-t)^\frac{m_1+m_1m_2-2}{m_2+2} dt\Biggl]^\frac{1}{2} \Biggl[\int^x_{y_c}(-t)^{-\frac{m_1+2m_2+m_1m_2+2}{m_2+2}}  dt\Biggl]^\frac{1}{2}\cdot
\end{align*}

Evaluating the integrals,
\begin{align*}
    G_L(x)=\frac{m_2+2}{m_1+m_2+m_1m_2}\Biggl[1-(-y_c)^{-\frac{m_1+m_2+m_1m_2}{m_2+2}}(-x)^\frac{m_1+m_2+m_1m_2}{m_2+2} \Biggl]^\frac{1}{2}.
\end{align*}

Thus,
\begin{align*}
    M_L=\sup_{a<x<b}G_L(x)=\frac{m_2+2}{m_1+m_2+m_1m_2}<+\infty\,,
\end{align*}
and $r(2,2)=2$. Therefore, 
$$M_L\leq C_L=\frac{2(m_2+2)}{m_1+m_2+m_1m_2}\leq r(2,2)M_L=\frac{2(m_2+2)}{m_1+m_2+m_1m_2}\cdot$$

In conclusion, $I=\int_{BC}(\omega_1+\omega_2)\geq0$ and, as mentioned in Step 2 the result follows. \end{proof}

\begin{theorem}\label{T5}
Let $\Omega_2$ be a Tricomi domain for the operator $\mathcal{O}$, which is $D$-star-shaped where $D=-(m_1+2)x\partial_x-(m_2+2)y\partial_y$ with $m_1$ odd and $\frac{m_2}{2}\in\mathbb{N}$. Let $u\in C^2(\overline{\Omega_2})$ be a regular solution of 

\begin{equation}
 \begin{cases}
\mathcal{O}(u):=-y^{m_1}u_{xx}-x^{m_2}u_{yy}=u|u|^{\alpha-1} \,\mbox{in} \quad\Omega_2, 
\\ \quad \quad u=0 \quad \quad \mbox{on} \quad AC\cup\sigma_2\subseteq\partial\Omega_2,
\end{cases}
\end{equation}
with $\alpha>2^*(m_1,m_2)-1=\frac{m_1+m_2-m_1m_2+8}{m_1+m_2+m_1m_2}$. Then $u\equiv0$ a.e. in $\Omega_2$.
\end{theorem}

\begin{proof}
The proof will be analogous to the proof of the Theorem \ref{T4}. 
\vspace{0,3cm}

\textbf{Step 1.} The same Step 1 of the Theorem \ref{T4}.
\vspace{0,3cm}

\textbf{Step 2.} Applying the Poho\^{z}aev-type identity \eqref{PI2}, we gain
\begin{align*}
    (m_1+m_2+4)\int_{\Omega_2} F(u)-\frac{m_1+m_2+m_1m_2}{2}&\int_{\Omega_2} uf(u)= \nonumber\\ &\frac{1}{2}\left[\int_{BC'\cup\sigma_2}\omega_1+\int_{BC'}\omega_2\right].
\end{align*}

\textbf{Step 3.} Analogous to Step 3 of the Theorem \ref{T4}, we have
\begin{align*}
 \int_{\sigma_2}\omega_1ds\geq 0.
\end{align*}

\textbf{Step 4.} Analogous to Step 4 of the Theorem \ref{T4}, we get 
\begin{align*}
    I=&2(m_2+2)\Bigl(\frac{m_2+2}{m_1+2}\Bigl)^\frac{m_2}{m_2+2}\int^0_{y_{c'}}(-t)^\frac{m_1+2m_2+m_1m_2+2}{m_2+2}(\phi'(t))^2\, dt \nonumber\\ 
    &-\frac{1}{2}\frac{(m_1+m_2+m_1m_2)^2}{m_2+2}\Bigl(\frac{m_2+2}{m_1+2}\Bigl)^\frac{m_2}{m_2+2}\int^0_{y_{c'}}(-t)^\frac{m_1+m_1m_2-2}{m_2+2}(\phi(t))^2 \, dt \, ,
\end{align*}

\textbf{Step 5.} The same Step 5. of the Theorem \ref{T4}.
\end{proof}

The following result arises from the difficulties mentioned in the introduction and it can be seen as a weak result of a possible generalization which does not have a final conclusion.  We do not have a similar result to the Theorems \ref{T4} and \ref{T5} on $\Omega_3$ for the Problem \eqref{P3}. But, we can consider the characteristic triangle $\Omega_4$ as being $\Omega_3$ for the Problem \eqref{P3} with $\sigma_3=\{(x,y)\in\mathbb{R}^2; \, x=0 \, \, \mbox{and} \, \, 2y_0 \leq y\leq 0 \}$. See figure \ref{TrianCarac}. 

\begin{figure}[!ht]
    \centering
    \includegraphics[width=7.5cm,angle=-90]{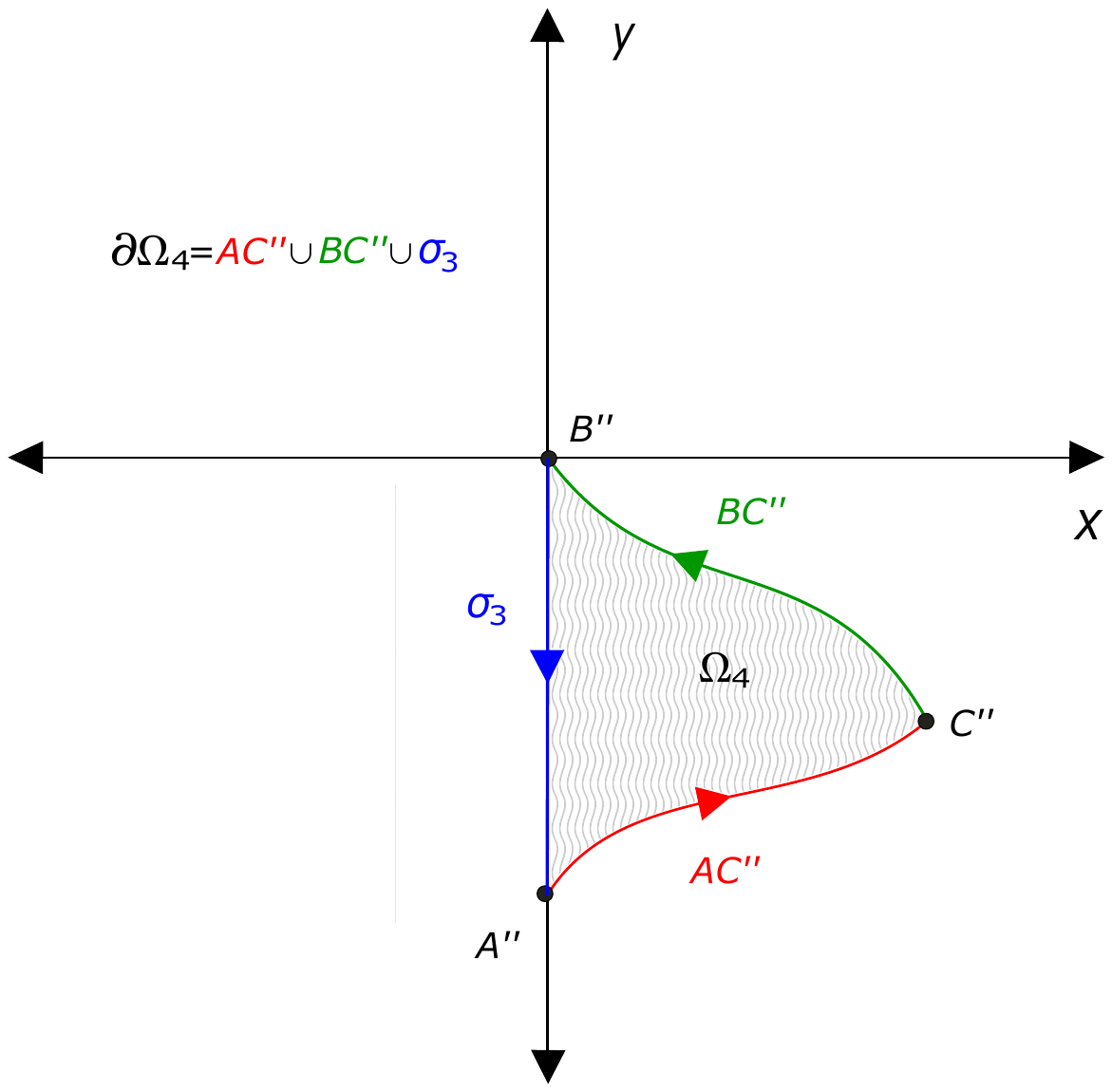}
    \caption{Domain $\Omega_4$ for the problem \eqref{P3}.}
    \label{TrianCarac}
\end{figure}

This idea comes from considering $\sigma_1=\{(x,y)\in\mathbb{R}^2; \, y=0 \, \, \mbox {and} \, 2x_0 \leq x\leq 0 \}$ for the Problem \eqref{P1} and $\sigma_2=\{(x,y)\in\mathbb{R}^2; \, y=0 \, \, \mbox {and} \, \, 0\leq x \leq 2x_0 \}$ for the Problem \eqref{P2} so, we could obtain two corollaries for the Theorems \ref{T4} and \ref{T5}.

\begin{theorem}\label{T6}
Let $\Omega_4$ be a Tricomi domain for the operator $\mathcal{O}$, which is $D$-star-shaped with $D=-(m_1+2)x\partial_x-(m_2+2)y\partial_y$ and $m_1$ odd and $\frac{m_2}{2}\in\mathbb{N}$. Let $u\in C^2(\overline{\Omega_4})$ be a regular solution of 

\begin{equation}
 \begin{cases}
\mathcal{O}(u):=-y^{m_1}u_{xx}-x^{m_2}u_{yy}=u|u|^{\alpha-1} &\mbox{in} \quad\Omega_4, 
\\ \quad \ \ \;u=0 \, &\mbox{on} \quad AC\cup\sigma_3\subseteq\partial\Omega_4,
\end{cases}
\end{equation}
with $\alpha>2^*(m_1,m_2)-1=\frac{m_1+m_2-m_1m_2+8}{m_1+m_2+m_1m_2}$. Then $u\equiv0$ a.e. in $\Omega_4$.
\end{theorem}

\begin{proof} 
The proff will be analogous to the proof of the Theorem \ref{T4}. 
\vspace{0,3cm}

\textbf{Step 1.} The same Step 1 of the Theorem \ref{T4}.
\vspace{0,3cm}

\textbf{Step 2.} Applying the Poho\^{z}aev-type identity \eqref{PI3}, we obtain
\begin{align*}
    (m_1+m_2+4)\int_{\Omega_4} F(u)-\frac{m_1+m_2+m_1m_2}{2}&\int_{\Omega_4} uf(u)= \nonumber\\ &\frac{1}{2}\left[\int_{BC''\cup\sigma_3}\omega_1+\int_{BC''}\omega_3\right].
\end{align*}

\textbf{Step 3.} Analogous to Step 3 of the Theorem \ref{T4}, we have 

\begin{align*}
 \int_{\sigma_3}\omega_1ds =0.
\end{align*}

\textbf{Step 4.} Analogous to Step 4 of the Theorem \ref{T4}, we get
\begin{align*}
    &I=2(m_2+2)\Bigl(\frac{m_2+2}{m_1+2}\Bigl)^\frac{m_2}{m_2+2}\int^0_{y_{c''}}(-t)^\frac{m_1+2m_2+m_1m_2+2}{m_2+2}(\phi'(t))^2\, dt \nonumber\\ 
    &-\frac{1}{2}\frac{(m_1+m_2+m_1m_2)^2}{m_2+2}\Bigl(\frac{m_2+2}{m_1+2}\Bigl)^\frac{m_2}{m_2+2}\int^0_{y_{c''}}(-t)^\frac{m_1+m_1m_2-2}{m_2+2}(\phi(t))^2 \, dt \, ,
\end{align*}

\textbf{Step 5.} The same Step 5. of the Theorem \ref{T4}.
\end{proof}




\section*{Acknowledgements}
    C. Reyes Peña is doctoral student in the Graduate Program in Mathematics at Federal University of S\~{a}o Carlos (PPGM-UFSCar) (2019-2024). This study was financed in part by the Coordinação de Aperfeiçoamento de Pessoal de Nivel Superior Brasil (CAPES) Finance Code 001. O H Miyagaki was supported in part by CNPq Nº 303256/2022-2 and FAPESP Nº 2020/16407-1.

\nolinenumbers


\end{document}